\documentclass[parskip=never,abstracton]{scrartcl}

\usepackage{csquotes}

\usepackage{mathrsfs}
\usepackage{yfonts}
\usepackage[T1]{fontenc}
\usepackage[utf8]{inputenc}
\usepackage{lmodern}
\usepackage{textcomp}
\usepackage{setspace}
\usepackage[a4paper,left=2.5cm,right=2cm,top=2.5cm,bottom=2.5cm]{geometry}
\usepackage{tgothic}

\usepackage{etoolbox}
\usepackage{xparse}

\usepackage{mathtools}
\usepackage{amsthm,thmtools}
\usepackage{dsfont}
\let\mathbb\mathds
\usepackage{amssymb}
\usepackage[scr]{rsfso} 
\usepackage{bm}
\usepackage{euscript}
\usepackage{amsbsy}
\DeclareMathAlphabet\mathbfcal{OMS}{cmsy}{b}{n}

\usepackage{microtype}
\usepackage{authblk}
\usepackage{graphicx} 
\usepackage{tikz}
\usetikzlibrary{cd,quotes}
\usetikzlibrary{decorations.markings}
\usepackage{pgfplots}
\pgfplotsset{compat=1.13}
\usepackage{booktabs}
\usepackage{float}
\usepackage{environ}
\usepackage{xcolor,soul}
\usepackage[shortlabels]{enumitem}
\usepackage{eso-pic}

\usepackage[normalem]{ulem}

\usepackage{hyperref} 
\hypersetup{%
	colorlinks,%
	linkcolor={red!60!black},%
	citecolor={blue!50!black},%
	urlcolor={blue!80!black}%
}

\usepackage{tikz-cd}
\tikzset{Rightarrow/.style={double equal sign distance,>={Implies},->},
	triple/.style={-,preaction={draw,Rightarrow}},
	quadruple/.style={preaction={draw,Rightarrow,shorten >=0pt},shorten >=1pt,-,double,double
		distance=0.2pt}}

\def\on{\operatorname}

\def\CC{\mathbb{C}}

\def\C{\EuScript{C}}
\def\D{\EuScript{D}}

\def\DD{\mathbb{D}}

\def\OO{\mathbb{O}}
\def\Fun{\on{Fun}}
\def\Cat{\on{Cat}}
\def\iCat{\EuScript{C}\!\on{at}}

\def\Hom{\on{Hom}}
\def\D{\EuScript{D}}
\def\Nerv{\on{N}}

\def\id{\on{id}}
\def\Set{\on{Set}}

\SetEnumitemKey{axioms}{itemsep=0pt,align=left,itemindent=!,
                        listparindent=\parindent,parsep=\parskip,
                        before=,
                        label={\Roman*}}

\SetEnumitemKey{claims}{itemsep=0pt,align=left,itemindent=!,
                        listparindent=\parindent,parsep=\parskip,
                        before=,
                        label={\Roman*}}

\DeclarePairedDelimiterX\set[1]{\lbrace}{\rbrace}{\def\given{\:\delimsize\vert\allowbreak\:}#1}
\newlist{implications}{description}{1} 
\setlist[implications]{itemsep=0pt,leftmargin=\parindent}
\NewDocumentCommand\implication{o}
  {\IfValueTF{#1}
    {\auximplication#1\relax}
    {\item[\normalfont($\,\Rightarrow\,$)]}}
\NewDocumentCommand\auximplication{u-u\relax}
  {\item[\normalfont(#1)$\,\Rightarrow\,$(#2)]}

\makeatletter
\newcounter{diagram}[section]
\def\thediagram{\thesection.\arabic{diagram}}
\def\ftype@diagram{4}
\def\ext@diagram{diag}
\def\fnum@diagram{Diagram~\thediagram}
\def\fs@diagram{htbp!}
\ExplSyntaxOn
\NewDocumentEnvironment{diagram}{O{htbp!}m}
  {\@float{diagram}[#1]\centering}
  {
   
   \caption{}
   \label{#2}
   \end@float
  }

\newcounter{subdiagram}[diagram]
\def\thesubdiagram{\thediagram.\arabic{subdiagram}}
\NewEnviron{multidiagram}{\expandafter\domultidiagram\BODY\enddomultidiagram}
\NewDocumentCommand\domultidiagram{omu\enddomultidiagram}
 {
  \IfValueTF{#1}{\diagram[#1]}{\diagram}{}
   \refstepcounter{diagram}
   \centering
    \seq_clear:N \l_tmpb_seq
    \seq_set_split:Nnn \l_tmpa_seq { \next } { #3 }
    \seq_map_inline:Nn \l_tmpa_seq
     {
      \seq_put_right:Nn \l_tmpb_seq
       {
        \begin{tabular}[b]{@{}c@{}}
         ##1 \\[3ex]
         \refstepcounter{subdiagram}
         \label{#2\othercolon\the\value{subdiagram}}
         Diagram~\thesubdiagram 
        \end{tabular}
       }
     }
    \seq_use:Nn \l_tmpb_seq { \qquad }
   \let\label\@gobble
   \let\caption\@gobble
  \enddiagram
 }
\ExplSyntaxOff
\def\othercolon{:}
\makeatother

\declaretheoremstyle[bodyfont=\itshape,notefont=\bfseries]{abellanA}
\declaretheoremstyle[notefont=\bfseries]{abellanB}

\declaretheorem[style=abellanA,numberwithin=section,name={Theorem}]{theorem}

\declaretheorem[style=abellanA,numberlike=theorem,name={Lemma}]{lemma}
\declaretheorem[style=abellanB,numberlike=theorem,name={Definition}]{definition}
\declaretheorem[style=abellanB,numberlike=theorem,name={Remark}]{remark}
\declaretheorem[style=abellanB,numberlike=theorem,name={Construction}]{construction}
\declaretheorem[style=abellanA,numberlike=theorem,name={Proposition}]{proposition}
\declaretheorem[style=abellanB,numberlike=theorem,name={Example}]{example}
\declaretheorem[style=abellanB,numbered=no,name={Notation}]{notation}
\declaretheorem[style=abellanB,numbered=no,name={Convention}]{convention}
\declaretheorem[style=abellanA,numberlike=theorem,name={Corollary}]{corollary}

\docsvlist{theorem,lemma,definition,remark,proposition,example,corollary}

\let\leq\leqslant
\let\geq\geqslant
\let\epsilon\varepsilon
\let\isom\cong

\newcommand*\numberset{\mathbb}
\newcommand*\tensor{\otimes}

\newcommand*\N{\numberset{N}}

\newcommand*\restr[3][\relax]{#2#1\rvert_{#3}^{}}
\newcommand*\mathblank{\mathord{-}}

\def\msSet{{\on{Set}_{\Delta}^+}}
\def\scsSet{{\on{Set}_{\Delta}^{\on{sc}}}}

\def\Cat{\on{Cat}}

\usepackage[bbgreekl]{mathbbol}

\let\emptyset\varnothing

\makeatletter
\newcommand{\fixed@sra}{$\vrule height 2\fontdimen22\textfont2 width 0pt\rightarrow$}
\newcommand{\shortarrowup}[1]{%
	\mathrel{\text{\rotatebox[origin=c]{65}{\fixed@sra}}}
}
\newcommand{\shortarrowdown}[1]{%
	\mathrel{\text{\rotatebox[origin=c]{250}{\fixed@sra}}}
}
\makeatother

\tikzset{boldred/.style = {draw=red, 
		line width=#1, -{Straight Barb[length=3pt]}},
	boldred/.default=2pt
}
\definecolor{darkmagenta}{rgb}{0.55, 0.0, 0.55}
\tikzset{boldmagenta/.style = {draw=darkmagenta, 
		line width=#1, -{Straight Barb[length=3pt]}},
	boldmagenta/.default=2pt
}				
\tikzset{boldblue/.style = {draw=blue, 
		line width=#1, -{Straight Barb[length=3pt]}},
	boldblue/.default=2pt
}

\tikzcdset{
  arrows={},
  nat/.style={Rightarrow},
  nat1/.style={nat,shift left=.8ex},
  nat2/.style={nat,shift right=.8ex,swap},
  map/.style={},
  map1/.style={map,shift left=.7ex},
  map2/.style={map,shift right=.7ex,swap},
  mapA/.style={map,shift left=.5ex},
  mapB/.style={map,shift left=.5ex},
  unique/.style={line cap=round,densely dotted},
  incl/.style={hookrightarrow},
  mono/.style={rightarrowtail},
  epi/.style={twoheadrightarrow},
  iso/.style={map,sloped,"{\tikz[x=.6ex,y=.3ex]{\draw[-](0,0)sin(1,1)cos(2,0)sin(3,-1)cos(4,0);}}"},
  line/.style={dash},
  line1/.style={line,shift left=.3ex},
  line2/.style={line,shift right=.3ex,swap},
}

\DeclareMathOperator\Nsc{N^{sc}}
\DeclareMathOperator*\colim{colim}
\DeclareMathOperator*\holim{holim}
\DeclareMathOperator*\hocolim{hocolim}
\def\op{{\on{op}}}

\def\Tw{\on{Tw}}
\def\tw{\on{Tw}}

\makeatletter
\newcommand*\dirlim{\mathop{\mathpalette\varlim@{\rightarrowfill@\scriptscriptstyle}}\nmlimits@}
\newcommand*\prolim{\mathop{\mathpalette\varlim@{\leftarrowfill@\scriptscriptstyle}}\nmlimits@}
\makeatother

\newcommand{\nat}{\Rightarrow}

\tikzset{
  abellanarrows/.style={line cap=round,line join=round,line width=.4pt},
  abellanarrowlength/.store in=\abellanarrowlength,
}

\ExplSyntaxOn

\cs_generate_variant:Nn \tl_replace_all:Nnn { Nf }

\NewDocumentCommand \func { s O{} m }
 {
  \group_begin:
   \IfBooleanTF{#1}
    { \keys_set:nn { abellan / func } { aligned = true , #2 } }
    { \keys_set:nn { abellan / func } {#2} }
   \abellan_func:n {#3}
  \group_end:
 }
\NewDocumentCommand \arr { s o m }
 {
  \IfBooleanF{#1}
   { \bool_if:NT \l_abellan_aligned_bool { & } }
  \abellan_arr:n {#3}
 }
\NewDocumentCommand \addarr { o m m }
 {
  \keys_set:nn { abellan / func / addarrow } { name = {#2} , #3 }
  \tl_clear:N \l_abellan_arrname_tl
 }
\NewDocumentCommand \setupfunc { m } { \keys_set:nn { abellan / func } {#1} }

\bool_new:N \l_abellan_aligned_bool
\tl_new:N \l_abellan_func_tl
\tl_new:N \l_abellan_arrname_tl
\tl_new:N \l_abellan_arrmode_tl
\tl_new:N \g_abellan_func_substitutions_tl
\quark_new:N \q_abellan

\keys_define:nn { abellan / func }
 {
  aligned .bool_set:N = \l_abellan_aligned_bool ,
  aligned .initial:n = false ,
  mode .choices:nn = { base , tikz } { \tl_set_eq:NN \l_abellan_arrmode_tl \l_keys_choice_tl } ,
  mode .initial:n = base ,
  tikzarrowlength .code:n = \tikzset{abellanarrowlength={#1}} ,
  tikzarrowlength .initial:n = 2em ,
  tikzarrowstyle .code:n = \tikzset{abellanarrows/.append ~ style={#1}} ,
 }
\keys_define:nn { abellan / func / addarrow }
 {
  name .tl_set:N = \l_abellan_arrname_tl ,
  base .code:n =
    \abellan_addarrow:nnww { \l_abellan_arrname_tl } { base } #1 \q_abellan ,
  tikz .code:n =
    \abellan_addarrow:nnww { \l_abellan_arrname_tl } { tikz } #1 \q_abellan ,
  substitutions .code:n =
   \tl_map_inline:nn {#1}
    {
     \tl_gput_right:Nx \g_abellan_func_substitutions_tl
      {
       \exp_not:n { \tl_replace_all:Nnn \l_abellan_func_tl {##1} }
        { \arr{\l_abellan_arrname_tl} }
      }
    } ,
 }
\cs_new_protected:Npn \abellan_func:n #1
 {
  \resetdynamicto
  \tl_set:Nn \l_abellan_func_tl {#1}
  \tl_replace_all:Nfn \l_abellan_func_tl
   { \char_generate:nn { `: } { 12 } } { \colon }
  \bool_if:NTF \l_abellan_aligned_bool
   { \tl_replace_all:Nnn \l_abellan_func_tl { ; } { \\ } }
   { \tl_replace_all:Nnn \l_abellan_func_tl { ; } { ,\ } }
  \tl_use:N \g_abellan_func_substitutions_tl
  \bool_if:NT \l_abellan_aligned_bool { \! \begin {aligned} \scan_stop: }
  \tl_use:N \l_abellan_func_tl
  \bool_if:NT \l_abellan_aligned_bool { \\ \end {aligned} }
 }
\NewDocumentCommand \abellan_addarrow:nnww { m m O{} u\q_abellan }
 {
  \exp_args:Nc \NewDocumentCommand { abellan_arr_#1_#2:w } { #3 }
   {
    \use:c { abellan_arr_ \l_abellan_arrmode_tl :n } { #4 }
   }
 }
\cs_new_protected:Npn \abellan_arr:n #1
 {
  \use:c { abellan_arr_#1_\l_abellan_arrmode_tl :w }
 }
\cs_new_protected:Npn \abellan_arr_base:n #1
 {
  \mathrel{#1}
 }
\cs_new_protected:Npn \abellan_arr_tikz:n #1
 {
  \mathrel{\vcenter{\hbox{\tikz[abellanarrows]{#1}}}}
 }
\ExplSyntaxOff

\addarr{monomorphism}
 {
  substitutions = {>->}{ mono }{\rightarrowtail}{\mono} ,
  base = [o]{\IfValueTF{#1}{\rightarrowtail{#1}}{\mono}} ,
  tikz = [o]{\draw[>->,inner sep=0pt]
     (0,0)--({max(\abellanarrowlength,.5em\IfValueT{#1}{+width("$\scriptstyle#1$")})},0)
     \IfValueT{#1}{node[midway,above=.7ex]{\smash{$\scriptstyle#1$}}
     node[midway,below=.7ex]{}}
     ;} ,
  }
\addarr{mapsto}
 {
  substitutions = {\mapsto}{ mapsto } ,
  base = \mapsto ,
  tikz = {\draw[|->](0,0)--(\abellanarrowlength,0);} ,
 }
\addarr{rrarrows} 
 {
  substitutions = {\rightrightarrows}{ doublerr }{\doublerr} ,
  base = \rightrightarrows ,
  tikz = {\draw[>={To[width=.8ex,length=.4ex]},->](0,0)--(\abellanarrowlength,0);
          \draw[yshift=.8ex,>={To[width=.8ex,length=.4ex]},->](0,0)--(\abellanarrowlength,0);} ,
 }
\addarr{inclusion}
 {
  substitutions = {\hookrightarrow}{\incl}{ incl } ,
  base = [o]{\IfValueTF{#1}{\hookrightarrow{#1}}{\incl}} ,
  tikz = {\draw[{Hooks[right]}->](0,0)--(\abellanarrowlength,0);} ,
 }
\addarr{natural}
 {
 substitutions = {\Rightarrow}{=>}{ nat }{\nat} ,
  base = [o]{\IfValueTF{#1}{\xRightarrow{#1}}{\nat}} ,
  tikz = [o]{\draw[line cap=butt,double equal sign distance,-Implies,inner sep=0pt]
    (0,0)--({max(\abellanarrowlength,.3em\IfValueT{#1}{+width("$\scriptstyle#1$")})},0)
    \IfValueT{#1}{node[midway,above=.7ex]{\smash{$\scriptstyle#1$}}
    node[midway,below=.7ex]{}}
    ;} ,
 }
\addarr{epimorphism} 
 {
  substitutions = {\twoheadrightarrow}{->>}{ epi }{\epi},
   base = [o]{\IfValueTF{#1}{\xtwoheadrightarrow{#1}}{\epi}} ,
   tikz = [o]{\draw[->>,inner sep=0pt]
      (0,0)--({max(\abellanarrowlength,.5em\IfValueT{#1}{+width("$\scriptstyle#1$")})},0)
      \IfValueT{#1}{node[midway,above=.7ex]{\smash{$\scriptstyle#1$}}
      node[midway,below=.7ex]{}}
      ;} ,
   }
\addarr{to}
 {
  substitutions = {\to}{\rightarrow}{ to }{ funct }{ map }{->} ,
  base = [o]{\IfValueTF{#1}{\xrightarrow{#1}}{\to}} ,
  tikz = [o]{\draw[->,inner sep=0pt]
    (0,0)--({max(\abellanarrowlength,.5em\IfValueT{#1}{+width("$\scriptstyle#1$")})},0)
    \IfValueT{#1}{node[midway,above=.7ex]{\smash{$\scriptstyle#1$}}
    node[midway,below=.7ex]{}}
    ;} ,
 }
\newcommand*\resetdynamicto
  {\gdef\dynamicto{\arr*{to}\gdef\dynamicto{\arr*{mapsto}}}}
\resetdynamicto

\setupfunc{mode=tikz}

\ExplSyntaxOn
\NewDocumentCommand \printheader { m o m }
 {
  \par\noindent
  \begin{minipage}[t]{\textwidth}\noindent
  
  \begin{tabular}[t]{ll}
     & \keyval_parse:NNn \abellan_printname:n \abellan_printnamemail:nn { #3 } 
  \end{tabular}
  \vspace{.4cm}
  \end{minipage}
  \begin{center}\Large\bfseries
   #1 \IfValueT{#2}{\\[1ex] \large #2}
  \end{center}
  \vspace{.6cm}
 }
\tl_new:N \g_abellan_autores_tl
\tl_new:N \g_abellan_asignatura_tl
\cs_new_protected:Npn \abellan_printnamemail:nn #1 #2
 {
  #1 \\ & \quad\texttt{#2} \\ &
 }
\cs_new_protected:Npn \abellan_printname:n #1
 {
  #1 \\ &
 }
\ExplSyntaxOff

\makeatletter
\tikzcdset{
	open/.code     = {\tikzcdset{hook, circled};},
	open'/.code    = {\tikzcdset{hook', circled};},
	circled/.code  = {\tikzcdset{markwith = {\draw (0,0) circle (.375ex);}};},
	markwith/.code ={
		\pgfutil@ifundefined%
		{tikz@library@decorations.markings@loaded}%
		{\pgfutil@packageerror{tikz-cd}{You need to say %
				\string\usetikzlibrary{decorations.markings} to use arrows with markings}{}}{}%
		\pgfkeysalso{/tikz/postaction = {
				/tikz/decorate,
				/tikz/decoration={markings, mark = at position 0.5 with {#1}}}
		}
	},
}
\makeatother

\def\scr{\EuScript}

\makeatletter
\newcommand*{\boxwedge}{%
	\mathbin{%
		\mathpalette\@boxwedge{}%
	}%
}
\newcommand*{\@boxwedge}[2]{%
	\sbox0{$#1\boxplus\m@th$}%
	\dimen2=.5\dimexpr\wd0-\ht0-\dp0\relax 
	\dimen@=\dimexpr\ht0+\dp0\relax
	\def\lw{.06}
	\kern\dimen2 
	\tikz[
	line width=\lw\dimen@,
	line join=round,
	x=\dimen@,
	y=\dimen@,
	]
	\draw
	(\lw/2,0) rectangle (1-\lw,1-\lw)
	(\lw,0) -- (.5,1-\lw-\lw/2) -- (1-\lw-\lw/2 ,0)
	;%
	\kern\dimen2 
}
\makeatletter

\newcommand{\goth}{\textfrak}

\title{Enhanced twisted arrow categories}
\author{Fernando Abell\'an Garc\'ia and Walker H. Stern}
\date{}

\begin{document}
    \maketitle
\begin{abstract}
	Given an $\infty$-bicategory $\DD$ with underlying $\infty$-category $\D$, we construct a Cartesian fibration $\Tw(\DD)\to \D \times \D^{\op}$, which we call the \emph{enhanced twisted arrow} $\infty$-category, classifying the restricted mapping category functor $\on{Map}_{\DD}:\D^{\op}\times \D \to \DD^{\op} \times \DD \to \iCat_{\infty}$. With the aid of this new construction, we provide a description of the $\infty$-category of natural transformations $\on{Nat}(F,G)$ as an end for any functors $F$ and $G$ from an $\infty$-category to an $\infty$-bicategory. As an application of our results, we demonstrate that the definition of weighted colimits presented in \href{https://arxiv.org/abs/1501.02161}{arXiv:1501.02161} satisfies the expected 2-dimensional universal property.
\end{abstract}

\tableofcontents

\section*{Introduction}
\addcontentsline{toc}{section}{Introduction}

Of the many tools belonging to the study of categories, perhaps the most key is the Yoneda lemma. The fully faithfulness of the functor 
\[
\func*{ \C \to \Set_\C; 
x \mapsto h_x:=\Hom_{\C}(-,x)}
\]
means, in particularly, that we can view functors $\func{f:\C^\op\to \Set}$ as \emph{universal properties}, and thereby uniquely specify an object $x$ by requiring $h_x\cong f$. 

In the higher-categorical realm, the good news is that this result still holds. The $(\infty,1)$-categorical Yoneda embedding 
\[
\func{\scr{Y}:\C \to \scr{S}_\C}
\] 
is fully faithful (c.f. e.g. \cite[5.1.3.1]{HTT}). While this is auspicious for the study of universal properties as described above, it comes with a significant complication. The standard presentation of the target category $\scr{S}_{\C}$ (which is also written variously as $\scr{P}(\C)$ or $\Fun(\C^\op,\scr{S})$) is in terms of a model structure on the category $\Fun(\mathfrak{C}[\C],\Set_\Delta)$ of simplicially enriched functors. 

The model $\Fun(\mathfrak{C}[\C],\Set_\Delta)$ is extremely useful in relating the underlying $\infty$-category to other $\infty$-categories --- for example in the proof of the $\infty$-categorical Yoneda lemma. The problem arises in that it is often extremely difficult to write down explicit simplicially-enriched functors, and explicit simplicially-enriched natural transformations between them. When the initial definition of $\C$ is as a quasi-category, it can even be difficult to write down $\mathfrak{C}[\C]$ explicitly. 

As in so many parts of higher category theory, the way out of this dilemma is the Grothendieck construction. We can proceed  according to the 
\begin{quote}
	{\scshape Slogan:} {\itshape Cartesian fibrations and maps between them are easier to work with than enriched functors and natural transformations between them.}
\end{quote}  
From this perspective, if we want to study representable functors and universal properties, we need first to classify the Yoneda embedding by a fibration. 

\subsection*{The twisted arrow category}
\addcontentsline{toc}{subsection}{The twisted arrow category}

The canonical solution to this problem is the \emph{twisted arrow ($\infty$-)category}. In e.g. \cite{LurieDAGX} and \cite{Cisinski}, it is shown that for each $\infty$-category $\C$, there is a right fibration\footnote{It is worth commenting that Cisinski and several other authors tend to work with the \emph{left} fibration associated to the same functor. The difference between the two definitions amounts to an \enquote{$\op$}, in the definition of the simplices of $\Tw(\C)$. Throughout the paper, we will only use the Cartesian/right fibration convention, and will omit any further mention of coCartesian/left fibrations.}
\[
\func{\Tw(\C)\to \C\times \C^\op}
\] 
which classifies the functor $\func{\Hom_{\C}:\C^\op\times \C\to \Cat_{\infty}}$.

The uses of the twisted arrow category are manifold. It appears, as suggested above, in the analysis of questions of representability throughout the higher categorical literature --- e.g. in \cite{LurieDAGX, HA}. In addition, it is used to explore $\mathbb{E}_k$-monoidal $\infty$-categories in \cite{HA}. In a completely different direction, there is a fundamental connection between twisted arrow categories and $\infty$-categories of spans/correspondences as described in, e.g. \cite[Ch. 10]{DKHSSI},\cite{BarwickK}, and \cite{Dualizing}. Moreover, this approach has been used to tackle questions related to $K$-theory in \cite{BarwickK}. 

The 1-simplices of $\Tw(\C)$ over a pair $(\alpha,\beta)$ comprising a  1-simplex in $\C\times \C^\op$ take the form of coherent diagrams 
\[
\begin{tikzcd}
 a\arrow[r,"f"]\arrow[d,"\alpha"'] & b \\
 a^\prime\arrow[r,"g"'] & b^\prime\arrow[u,"\beta"'] 
\end{tikzcd}
\] 
in $\C$. In practice, this means that that the fibers have $1$-simplices consisting of diagrams  
\[
\begin{tikzcd}
	a\arrow[r,bend left, "f"]\arrow[r,bend right, "g"'] & b
\end{tikzcd}
\]
that commute up to a chosen 2-cell, i.e. the morphisms in the fiber can be easily interpreted as two-cells $\func{f\nat[\sim] g}$ in $\C$. More generally, the $n$-simplices of $\Tw(\C)$ are given by maps $\func{\Delta^n\star (\Delta^n)^\op\to \C}$, and the projections to $\C$ and $\C^\op$ are induced by the inclusions $\begin{tikzcd}
\Delta^n\arrow[r] & \Delta^n\star(\Delta^n)^\op & (\Delta^n)^\op\arrow[l]
\end{tikzcd}$. 

\subsection*{Towards an enhanced twisted arrow category}
\addcontentsline{toc}{subsection}{Towards an enhanced twisted arrow category}

Given an $\infty$-bicategory $\CC$, presented as a fibrant scaled simplicial set, our aim will be to construct an $\infty$-category $\Tw(\CC)$ together with a Cartesian fibration $\Tw(\CC)\to \C\times \C^\op$ which classifies the composite functor
\[
\func{\C^\op \times \C \to \CC^\op \times \CC \to \goth{C}\!\on{at}_{\infty},}
\]
where $\goth{C}\!\on{at}_{\infty}$ is the $(\infty,2)$-category of $\infty$-categories. The first step towards this construction is to decide what the 1-simplices of $\Tw(\CC)$ should be. We would still like these to be something like diagrams 
\[
\begin{tikzcd}
a\arrow[r,"f"]\arrow[d,"\alpha"'] & b \\
a^\prime\arrow[r,"g"'] & b^\prime\arrow[u,"\beta"'] 
\end{tikzcd}
\] 
in $\CC$, e.g. 3-simplices. 

When $\alpha$ and $\beta$ are identities, we would like these 3-simplices to encode precisely the choice of a 2-morphism $\func{f\nat g}$. However, heuristically such a 3-simplex should, in fact, encode two factorizations: 
\[
\begin{tikzcd}[sep=3em]
a\arrow[r,"f"{name=bar}]\arrow[d,"\id"']\arrow[dr,""{name=foo}] & b \\
|[alias=X]|a^\prime\arrow[r,"g"'] & |[alias=Y]|b^\prime\arrow[u,"\id"'] \arrow[Rightarrow, from=foo,
to=X, shorten <=.1cm,shorten >=.1cm] \arrow[Rightarrow, from=bar,
to=Y, shorten <=.3cm,shorten >=.3cm]
\end{tikzcd} \qquad \qquad 
\begin{tikzcd}[sep=3em]
a\arrow[r,"f"{name=bar}]\arrow[d,"\id"'] & b \\
|[alias=X]|a^\prime\arrow[r,"g"']\arrow[ur,""{name=foo}] & |[alias=Y]|b^\prime\arrow[u,"\id"'] \arrow[Rightarrow, from=bar,
to=X, shorten <=.3cm,shorten >=.3cm] \arrow[Rightarrow, from=foo,
to=Y, shorten <=.1cm,shorten >=.1cm]
\end{tikzcd}
\]
together with the 3-simplex itself, which indicates that the composites --- 2-morphisms $\func{f\nat g}$ --- of both factorizations are equivalent. Fortunately, in the realm of scaled simplicial sets, we can declare certain 2-simplices to be `thin' --- i.e., declare the corresponding 2-morphisms to be invertible. With this in mind, we can force half of each factorization to be invertible 
\[
\begin{tikzcd}[sep=3em]
a\arrow[r,"f"{name=bar}]\arrow[d,"\id"']\arrow[dr,""{name=foo}] & b \\
|[alias=X]|a^\prime\arrow[r,"g"'] & |[alias=Y]|b^\prime\arrow[u,"\id"'] \arrow[phantom, from=foo,
to=X,"\circlearrowleft" description] \arrow[Rightarrow, from=bar,
to=Y, shorten <=.3cm,shorten >=.3cm]
\end{tikzcd} \qquad \qquad 
\begin{tikzcd}[sep=3em]
a\arrow[r,"f"{name=bar}]\arrow[d,"\id"'] & b \\
|[alias=X]|a^\prime\arrow[r,"g"']\arrow[ur,""{name=foo}] & |[alias=Y]|b^\prime\arrow[u,"\id"'] \arrow[Rightarrow, from=bar,
to=X, shorten <=.3cm,shorten >=.3cm] \arrow[phantom, from=foo,
to=Y, "\circlearrowleft" description]
\end{tikzcd}
\]
In this case, we obtain two 2-morphisms $\func{f\nat g}$ and a 3-simplex showing that they are equivalent --- precisely the data that we would like.\footnote{On thing we are glossing over is why we choose the \enquote{lower} 2-simplices as thin, rather than the \enquote{upper} ones. In a nutshell, the reason is that the lower 2-simplices will encode composites, and thus be unique up to contractible choice.} 

This suggests a trial definition for the twisted arrow category of an $\infty$-bicategory.
\begin{quote}
	The twisted arrow $\infty$-bicategory $\mathbb{T}\!\!\on{w}(\CC)$ should have $n$-simplices 
	\[
	\mathbb{T}\!\!\on{w}(\CC)_n:= \Hom_{\scsSet}((\Delta^n\star(\Delta^n)^\op,T),\CC)
	\]
	where $T$ is the scaling given by requiring that, under the identification $\Delta^n\star(\Delta^n)^\op\cong \Delta^{2n+1}$, the simplices $\{i,j,2n+1-j\}$ and $\{j,2n+1-j, 2n+1-i\}$ are thin for $i<j$. 
\end{quote}
However, we would expect such a construction to yield a fibration over the $(\infty,2)$-category $\CC\times \CC^\op$. There are technical difficulties to such a definition, not least the fact that the corresponding Grothendieck construction has not yet appeared in the literature. While we expect this definition to yield a genuine $(\infty,2)$ twisted arrow category, we will restrict ourselves to the examination of the induced functor $\func{\C^\op\times \C\to \Cat_\infty}$

To restrict to the fibration classifying this functor, we use the expected base-change properties of a hypothetical Cartesian Grothendieck construction\footnote{A likely candidate for the kind of fibration such a construction would involve is the \emph{outer Cartesian fibration} of \cite{HarpazEquivModels}.} over an $\infty$-bicategorical base. To wit, we define $\Tw(\CC)$ to be the pullback 
\[
\begin{tikzcd}
\Tw(\CC)\arrow[r]\arrow[d]\arrow[dr,phantom, "\lrcorner", very near start] & \mathbb{T}\!\!\on{w}(\CC)\arrow[d]\\
\C\times \C^\op\arrow[r] & \CC\times \CC^\op 
\end{tikzcd}
\]
In terms of the scaling on $\Delta^n\star (\Delta^n)^\op$, This pullback simply amounts to requiring that every 2-simplex contained within $\Delta^n$ and every 2-simplex contained within $(\Delta^n)^\op$ is thin. Using pushouts by scaled anodyne morphisms of the kind described in \cite[Rmk. 1.17]{HarpazEquivModels}, we can extend this scaling to consider a cosimplicial object $Q(n):=(\Delta^n\star (\Delta^n)^\op, T)$ in scaled simplicial sets, where the non-degenerate thin simplices of $T$ are: 
\begin{itemize}
	\item 2-simplices which factor through $\Delta^n$ or $(\Delta^n)^\op$. 
	\item 2-simplices $\Delta^{\{i,j, 2n+1-k\}}$ and $\Delta^{\{k,2n+1-j,2n+1-i\}}$ for $0\leq i\leq j\leq k \leq n$. 
\end{itemize}

This is the definition of $\Tw(\CC)$ we adopt throughout the present paper,  which is justified by the following result.

\begin{theorem}\label{thm:fibcomparisonintro}
	Let $\CC$ be an $\infty$-bicategory. Then $\Tw(\CC)\to \C \times \C^{\op}$ is a Cartesian fibration classifying the restricted mapping category functor
	\[
		\func{\on{Map}_{\CC}:\C^{\op} \times \C \to \CC^{\op} \times \CC \to \goth{C}\!\on{at}_{\infty}}
	\]
\end{theorem}
This is an amalgam of \autoref{thm:TwFibrant} and \autoref{thm:TwClassifiesMap} from the text.

\subsection*{Applications: The category of natural transformations as an end}
\addcontentsline{toc}{subsection}{Applications: The category of natural transformations as an end} 
Once verified that our definition enjoys the desired properties we turn into our main motivation for this paper: understanding the category of natural transformations $\on{Nat}(F,G)$ between functors from an $\infty$-category to an $\infty$-bicategory. To do so, we obtain that expected description of the category of natural transformations as an end.

\begin{theorem}\label{thm:natintro}
	Let $\C$ be a $\infty$-category and $\DD$ an $\infty$-bicategory. Then for every pair of functors $F,G: \C \to \DD$  there exists a  equivalence of $\infty$-categories
	\[
		\func{ \on{Nat}_{\C}(F,G) \to \lim_{\Tw(\C)^{\op}}\on{Map}_{\DD}(F(\mathblank),G(\mathblank))}
	\]
	which is natural in each variable.
\end{theorem}

This result allows us to analyze in greater detail the theory of weighted colimits of $\iCat_{\infty}$-valued functors exposed in \cite{GHN}, showing that this definition coincides with the definition provided by the first author in \cite{AG20}. The proof of this fact together with the results of \cite{AG20} constitute a partial answer to a series of conjectures involving $\infty$-bicategorical colimits and a categorified theory of cofinality introduced by the authors in \cite{AGS20}.

\subsection*{Structure of the paper}
\addcontentsline{toc}{subsection}{Structure of the paper}

The paper will be laid out as follows. We begin with a preliminary section, which lays out the notational conventions we follow, and explains several technical constructions and lemmata which we use throughout the paper. In particular, we give basic definitions for cosimplicial objects, state and prove a general lemma on subsets $K\subset \Delta^n_\dagger$ of a scaled $n$-simplex such that $\func{K\to \Delta^n_\dagger}$ is scaled anodyne, and define a structure on a poset sufficient for us to give a clean description of the simplicial mapping spaces in a quotient of its nerve. 

From there, the work starts in earnest. In \autoref{sec:TwDefnFib}, we give the formal definition of $\Tw(\CC)$, and prove that $\Tw(\CC)\to \C\times \C^\op$ is a Cartesian fibration, making use of the aforementioned lemma on simplicial subsets of scaled $n$-simplices. We then turn to \autoref{sec:classification}, in which we prove that this Cartesian fibration classifies precisely the enhanced mapping functor 
\[
\func{\C^\op \times \C\to \CC^\op\times \CC \to \goth{C}\!\on{at}_{\infty}.}
\]
This proof is highly technical, and freely uses results from \cite{LurieGoodwillie} and \cite{HarpazEquivModels}. 

In \autoref{sec:Nat}, our attention then turns to the true aim of the paper, a proof of the proposition that, given two functors $F,G:\C\to\DD$ from an $\infty$-category to an $(\infty,2)$-category, the $\infty$-category of natural transformations between them can be expressed as a limit 
\[
\on{Nat}(F,G)\simeq \lim_{\Tw(\C)^\op} \on{Map}_{\DD}(F(-),G(-)),
\]
i.e., an end. Once again the proof is highly technical, making use of a wide variety of techniques native to the contexts of scaled simplicial sets and marked simplicial sets. In particular, the proof relies heavily on a sort of d\'evissage --- one in which we reduce from the case of a general $\infty$-category (indeed, simplicial set) $\C$ to the cases $\C=\Delta^0$ and $\C=\Delta^1$. 

We conclude with applications of this theorem, where we upgrade several results appearing in \cite{GHN}. 

\subsection*{Acknowledgments}
\addcontentsline{toc}{subsection}{Acknowledgments}

F.A.G. would like to acknowledge the support of the VolkswagenStiftung through the LichtenbergProfessorship Programme while he conducted this research. W.H.S was supported by Universit\"at Hamburg during the early stages of this work, and by the NSF Research Training Group at the University of Virginia (grant number DMS-1839968) during the later stages.

\section{Preliminaries}
We begin by presenting some background information necessary for the paper, and proving some general lemmata which will help simplify the technical arguments in later sections. We will not, in general, recapitulate material from \cite{HTT} and \cite{LurieGoodwillie}, as doing so would greatly extend the length of the present document for dubious benefit. In particular, we will assume that the reader is familiar with the theories of quasi-categories, Cartesian fibrations, and scaled simplicial sets, as well as the attendant model structures. We will, however, briefly collect the notations and conventions we will use for these before embarking on the preliminaries proper. 

\begin{notation}[Model categories]
	We denote by $\Set_\Delta$ the category of simplicial sets, $\Set_\Delta^+$ the category of marked simplicial sets, and $\scsSet$ the category of scaled simplicial sets. We consider these to be equipped with the Joyal, Cartesian, and bicategorical model structures, respectively. Where context clarifies the meaning, an unadorned Latin capital --- e.g. $X$ --- may be used to denote an object of any of these categories. When it is necessary to specify a marking or a scaling on $X\in \Set_\Delta$, we do so by writing a superscript --- e.g. $X^\dagger$ --- for a marking, and a subscript --- e.g. $X_\dagger$ --- for a scaling. In particular, the subscripts $\sharp$ and $\flat$ will denote the maximal and minimal scalings, respectively.
\end{notation}

\begin{notation}[Rigidification]
	We denote by $\Cat_{\Delta}$ the category of simplicial set enriched categories, and by $\Cat_{\msSet}$ the category of marked simplicial set enriched categories. We denote by $\func{\mathfrak{C}:\Set_\Delta\to \Cat_\Delta}$ the rigidification functor, and by $\func{\mathfrak{C}^{\on{sc}}:\scsSet\to \Cat_{\msSet}}$ its scaled variant. In the presence of the sub- and superscript convention above, we will conventionally denote 
	\[
	\mathfrak{C}[X](x,y)^\dagger:=\mathfrak{C}^{\on{sc}}[X_\dagger](x,y)
	\]
	for any $x,y\in X$. 
\end{notation} 

\begin{convention}[Fibrant objects]
	By an $\infty$-category, we will mean an $(\infty,1)$-category, presented as either a quasi-category or a fibrant marked simplicial set. We will, wherever possible, use calligraphic capitals --- e.g. $\scr{C}$ --- for $\infty$-categories. 
	
	By an $\infty$-bicategory, we will mean an $(\infty,2)$-category presented as a fibrant scaled simplicial set.\footnote{The potential for confusion between $\infty$-bicategories and weak $\infty$-categories created by the terminology of \cite{LurieGoodwillie} is obviated by \cite[Thm. 5.1]{HarpazEquivModels}.} Where possible, we will denote $\infty$-bicategories by blackboard-bold capitals --- e.g. $\mathbb{C}$. 
\end{convention}

\subsection{Cosimplicial objects}
\begin{definition}
	Let $C$ be an ordinary 1-category. A functor $\func{F: \Delta \to C}$ will be called a cosimplicial object in $C$.
\end{definition}

\begin{notation}
	Given $[n] \in \Delta$ we will denote its image under $F$ by $F(n)$.
\end{notation}

In the following sections, we will make extensive use of cosimplicial objects with target a cocomplete category $C$. Namely, those that can be “freely extendend” by colimits. Indeed by taking the left Kan extension along the Yoneda embedding $\func{\mathcal{Y}:\Delta \to \on{Set}_{\Delta}}$ we can produce a pair of adjoint functors
\[
	\begin{tikzcd}
	\mathcal{Y}_{!}F:\operatorname{Set}_\Delta \arrow[r, shift left]    & C:F^{*} \arrow[l, shift left]
	\end{tikzcd}
\]
where for every $c \in C$ the $n$-simplices of $F^{*}(c)$ are given by maps $F(n) \to c$.

\begin{example}
	Let $C=\on{Set}_{\Delta}$ and let $X \in  \on{Set}_{\Delta}$. We define a cosimplicial object
	\[
		\func{(\mathblank)\times X: \Delta \to \on{Set}_{\Delta};[n] \mapsto \Delta^n \times X.}
	\]
	The right adjoint to this cosimplicial object sends each $\infty$-category $Y$ to the functor $\infty$-category $\on{Fun}(X,Y)$.\footnote{More generally, the right adjoint gives the internal hom of $\Set_\Delta$.}
\end{example}

\begin{notation}
	Let $C$ be a cocomplete category and $F$ a cosimplicial object on $C$. We set the following notation
	\[
		\partial F^n= \colim_{\Delta^I \to \partial \Delta^n}F(I).
	\]
\end{notation}

\subsection{Scaled anodyne maps from dull subsets}
\begin{definition}\label{def:dull}
	Let $\mathbf{P}(n)$ denote the power set of $[n]$ with $n\geq 0$. We say that $\mathcal{A} \subsetneq \mathbf{P}(n)$ is \emph{dull} if the following conditions are satisfied:
	\begin{enumerate}
		\item It does not contain the empty set, $\emptyset \notin \mathcal{A}$
		\item \label{dullcond:noi} There exists $0<i<n$ such that $i \notin S$ for every $S \in \mathcal{A}$.
		\item It contains a pair of singletons $\set{u},\set{v}\in \mathcal{A}$ such that $u<i<v$.
		\item \label{dullcond:nointer} For every $S,T \in \mathcal{A}$ it follows that $S \cap T =\emptyset$.	
	\end{enumerate}
	We will call the element $i$ in condition (\ref{dullcond:noi}), \emph{the pivot point}.
\end{definition}

\begin{definition}
	Let $\mathcal{A} \subsetneq \mathbf{P}(n)$ be a dull subset. Given an scaled $n$-simplex $\Delta^n_{\dagger}$, we define 
	\[
		\scr{S}^{\mathcal{A}}=\bigcup\limits_{S \in \mathcal{A}}\Delta^{[n]\setminus S} \subsetneq \Delta^n
	\]
	and denote $\scr{S}$ equipped with the induced scaling by $\scr{S}^{\mathcal{A}}_{\dagger}$. When the choice of dull subset is clear, we will use the abusive notation $\scr{S}_{\dagger}$.
\end{definition}

\begin{definition}
	Let $\mathcal{A}\subsetneq \mathbf{P}(n)$ be a dull subset. We call $X\in \mathbf{P}(n)$ an \emph{$\mathcal{A}$-basal set} if it contains precisely one element from each $S\in \mathcal{A}$. We denote the set of all $\mathcal{A}$-basal sets by $\on{Bas}(\mathcal{A})$.
\end{definition}

\begin{remark}
	Note that our definitions guarantee both that $\on{Bas}(\mathcal{A})\neq \emptyset$, and that all $\mathcal{A}$-basal sets have the same cardinality.
\end{remark}

\begin{definition}
	Given a dull subset $\mathcal{A}$, we define $\scr{M}_{\mathcal{A}}$ to be the set of subsets $X \in \mathbf{P}(n)$ satisfying the following conditions:
	\begin{itemize}
		\item[$A1)$] $X$ contains the pivot point, $i \in X$.
		\item[$A2)$] The simplex $\sigma_X:\Delta^X \to \Delta^n$ does not factor through $\scr{S}$.
	\end{itemize}
	We set $\kappa_{\mathcal{A}}:=\min\set{|X| \given X \in \mathcal{M}_{\mathcal{A}}}$ and define, for every $\kappa_{\mathcal{A}}\leq j \leq n$, the subset $\mathcal{M}_{\mathcal{A}}^{j}\subset \mathcal{M}_{\mathcal{A}}$ consisting of those sets of cardinality at most $j$. 
\end{definition}

\begin{remark}
	To ease the notation, when the choice of dull subset it is clear we will  drop the subscript $\mathcal{A}$ in $\mathcal{M}_{\mathcal{A}}$ and $\kappa_{\mathcal{A}}$.
\end{remark}

\begin{lemma}\label{lem:kappa}
	Let $\mathcal{A}$ be a dull subset of $\mathbf{P}(n)$ with pivot point $i$. Then it follows that  
	\[
		\mathcal{M}^{\kappa}=\set{X_0 \cup \set{i} \enspace \given X_0 \in \on{Bas}(\mathcal{A})}.
	\]
\end{lemma}
\begin{proof}
	Left as an exercise.
\end{proof}

\begin{notation}
	Let $\mathcal{A}\subsetneq \mathbf{P}(n)$ be a dull subset with pivot point $i$. Given an $\mathcal{A}$-basal set $X$, we will denote by $\ell_{i-1}^X$ $\ell_{i}^X$ the pair of consecutive elements in $X$ such that $\ell_{i-1}^X<i<\ell_{i}^X$. 
\end{notation}

\begin{lemma}[The pivot trick]\label{lem:pivot}
	Let $\mathcal{A}\subsetneq \mathbf{P}(n)$ be a dull subset with pivot point $i$, and let $\Delta^n_{\dagger}$ be an scaled simplex.  For  $Z\in \on{Bas}(\mathcal{A})$ suppose that the following condition holds.
	\begin{itemize}
		\item For every $r,s \in [n]$ such that $\ell_{i-1}^Z\leq r <i <s \leq \ell_{i}^Z$ the simplex $\set{r,i,s}$ is scaled in $\Delta^n_{\dagger}$.
	\end{itemize}
	Then $\scr{S}_{\dagger} \to \Delta^n_{\dagger}$ is scaled anodyne.
\end{lemma}
\begin{proof}
	For ease of notation we will drop the subscript denoting the scaling in this proof, assuming that all simplicial subsets are equipped with the scaling inherited from $\Delta^n_\dagger$. We define for $\kappa \leq j \leq n$
	\[
		Y_{j}=Y_{j-1}\cup \bigcup\limits_{X \in \mathcal{M}^j}\sigma_X,
	\]
	where we set $Y_{\kappa-1}=\scr{S}$. This yields a filtration
	\[
		\func{\scr{S}\to Y_{\kappa} \to \cdots \to Y_{n-1} \to \Lambda^n_i \to \Delta^n.}
	\]
	where $\Lambda^n_i=Y_n$. We will show that each step of this factorization is scaled anodyne. 
	
	Let $X \in \mathcal{M}^{j}$ with $\kappa \leq j \leq n-1$. Let us note that as a consequence \autoref{lem:kappa} we obtain a  pullback diagram
	\[
		\begin{tikzcd}[ampersand replacement=\&]
			\Lambda^X_{i} \arrow[d] \arrow[r] \arrow[dr,phantom,"\lrcorner",very near start] \& \Delta^X  \arrow[d,"\sigma_X"] \\
			Y_{j-1} \arrow[r]\&  Y_{j}
		\end{tikzcd}
	\]
	Additionally, the condition of the lemma guarantees that $i$ together with its neighboring elements in $\Delta^X$ form a scaled 2-simplex. Thus, the map $\func{\Lambda^X_i\to \Delta^X}$ is scaled anodyne, allowing us to add $\Delta^X$. It also follows from our definitions that given $X,Y \in \mathcal{M}^{j}$ such that $X \neq Y$ then $\sigma_X \cap \sigma_Y \in Y_{j-1}$, so that we can add the $j$-simplices $\Delta^X$ to $Y_{j-1}$ irrespective of their order. This shows that $Y_{j-1}\to Y_j$ is scaled anodyne.
\end{proof}

\begin{remark}\label{rmk:scaledanodyneduals}
	It is worth noting that the procedure outlined in \autoref{lem:pivot} only makes use of a special subset of the scaled anodyne maps: that generated by the inner horn inclusions 
	\[
	\func{\Lambda^n_i\to \Delta^n}
	\]
	where $\Delta^{\{i-1,i,i+1\}}$ is scaled. Significantly, while the class of scaled anodyne maps is not, in general, self-dual (i.e. $\func{f^\op: X^\op\to Y^\op}$ need not be scaled anodyne when $\func{f:X\to Y}$ is), the class generated by these scaled inner horn inclusions is. We will make use of this property to further simplify applications of \autoref{lem:pivot}.
\end{remark}

\subsection{Poset partitions}

In \autoref{sec:classification} it will be necessary for us to consider mapping spaces in quotients of nerves of posets, as well as their scaled analogues. While these mapping spaces are quite straightforward to describe, we here collect a number of descriptions and notations so as to better facilitate the flow of the later sections of the paper. 

\begin{definition}\label{defn:orderedpartitions}
	Let $J$ be a finite poset, and denote by $\scr{J}$ its nerve. We call a pair of subsets $J_0,J_1$ an \emph{ordered partition} of $J$ if the following three conditions are satisfied. 
	\begin{itemize}
		\item $J_0\cup J_1=J$.
		\item $J_0\cap J_1=\emptyset$.
		\item For every $x\in J_0$ and every $y\in J_1$, we have either $x<y$ or $x$ and $y$ are incomparable.  
	\end{itemize}  
	For such an ordered partition, we denote by $\scr{J}^R$ the quotient
	\[
	\scr{J}^R:=\scr{J}\coprod_{\Nerv(J_1)}\Delta^0, 
	\]
	and by $\widetilde{\scr{J}}$ the quotient 
	\[
	\widetilde{\scr{J}}:=\Delta^0\coprod_{\Nerv(J_0)}\scr{J}\coprod_{\Nerv(J_1)}\Delta^0.
	\]
	We denote the two objects of $\widetilde{\scr{J}}$ by $\ast_0$ and $\ast_1$, and denote the `collapse point' of $\scr{J}^R$ by $\ast_1$. 
\end{definition}

\begin{remark}
	Note that the definition of an ordered partition is symmetric --- the opposite of an ordered partition is still an ordered partition. It is for this reason that we only consider the quotient $\scr{J}^R$ and not some analogous $\scr{J}^L$ as well. 
\end{remark}

\begin{example}\label{ex:orderedpartofQ}
	In the sequel we will make extensive use of a cosimplicial object $Q(n):=\Delta^n\star (\Delta^n)^\op$. Each level of this cosimplicial object admits a canonical ordered partition. Under the identification $Q(n)\cong \Delta^{2n+1}=\Nerv([2n+1])$, this ordered partition is given by $J_0=[n]$ and $J_1=\{n+1,\ldots, 2n+1\}$. We will abusively denote each of these ordered partitions by $(J_0^Q,J_1^Q)$.
\end{example}

\begin{construction} 
	Given a finite poset $J$ and an ordered partition $(J_0,J_1)$, we construct a poset $P_{\scr{J}}$ as follows. The objects of $P_{\scr{J}}$ are totally ordered subsets $S\subset J$ such that $\min(S)\in J_0$ and $\max(S)\in J_1$, ordered by inclusion. We will denote the nerve by $\scr{P}_{\scr{J}}:=\Nerv(P_{\scr{J}})$.  
	
	Let $\underline{S}:=(S_0\subset\cdots S_k)$ be a $k$-simplex of $\scr{P}_{\scr{J}}$. Set $s_0^R:=\min (S_0\cap J_1)$. We define the \emph{right truncation} of $\underline{S}$ to be the simplex 
	\[
	\underline{S}^R:=(S_0^R\subset \cdots \subset S_k^R) 
	\]
	where $S_\ell^R:=\{s \in S_\ell\mid s\leq s_0^R\}$. We similarly define $s_0^L:=\max(S_0 \cap J_0)$ and its correspondig \emph{left truncation} $\underline{S}^L$ where $S_{\ell}^L:=\set{s \in S_{\ell} \given s \geq s_0^L}$. The \emph{ambidextrous truncation} $\underline{S}^A$ is obtained by taking both the left and right truncation of $S$. We can then define two equivalence relations on $\scr{P}_{\scr{J}}$. 
	\begin{enumerate}
		\item We say that $k$-simplices $\underline{S}$ and $\underline{T}$ are \emph{right equivalent}, and we write 
		\[
		\underline{S}\sim_R \underline{T}, 
		\] 
		when $\underline{S}^R=\underline{T}^R$. 
		\item We say that $k$-simplices $\underline{S}$ and $\underline{T}$ are \emph{ambi-equivalent}, and we write 
		\[
		\underline{S}\sim_A \underline{T}, 
		\] 
		when $\underline{S}^A=\underline{T}^A$. 
	\end{enumerate}  
	Note that both of these equivalence relations respect the face and degeneracy maps, so that the quotients of $\scr{P}_{\scr{J}}$ by $\sim_R$ and $\sim_A$ are simplicial sets. 
	
	Finally, for any $j\in J_0$, we define $P_{\scr{J}}^j\subset P_{\scr{J}}$ to be the full subposet on those sets $S$ with $\min(S)=j$.  Note that $\sim_R$ descends to an equivalence relation on $\scr{P}_{\scr{J}}^j$. 
\end{construction}

We can then characterize the desired mapping spaces of $\mathfrak{C}[\scr{J}]$ in terms of the above posets.

\begin{lemma}\label{lem:posetquotientmappingspaces}
	Let $J$ be a finite poset, and $(J_0,J_1)$ and ordered partition of $J$. Then 
	\begin{enumerate}
		\item for every $j\in J_0$ there is an isomorphism 
		\[
		\mathfrak{C}[\scr{J}^R](j,\ast_1)\cong (\scr{P}_{\scr{J}}^j)_{/\sim_R}.
		\]
		\item There is an isomorphism 
		\[
		\mathfrak{C}[\widetilde{\scr{J}}](\ast_0,\ast_1)\cong (\scr{P}_{\scr{J}})_{/\sim_A}.
		\]
	\end{enumerate}
\end{lemma}

\begin{proof}
	Follows from, e.g., the necklace characterization of \cite{DuggerSpivak} once the definitions have been unwound.
\end{proof}

\section{The enhanced twisted arrow category}\label{sec:TwDefnFib}

Our construction of the enhanced twisted arrow category will depend on an upgrade of the cosimplicial object
\[
\func{\Delta\to \Set_\Delta; [n]\mapsto \Delta^n\star (\Delta^n)^\op} 
\]
to a cosimplicial object in scaled simplicial sets. For a discussion of the intuition behind this choice of scaling, see the introduction. To simplify some of the discussion to come, we introduce some notational conventions surrounding $\Delta^n\star(\Delta^n)^\op$. Note, before we begin, that there is a canonical identification $\Delta^n\star(\Delta^n)^\op\cong \Delta^{2n+1}$, which we will often use without comment. 

\begin{notation}
	In general, we will denote elements of $\Delta^n\star(\Delta^n)^\op$ by $i\in \Delta^n$ or $\overline{i} \in (\Delta^n)^\op$. 
	Note that under the identification $\Delta^n\star(\Delta^n)^\op\cong \Delta^{2n+1}$, $\overline{i}$ is identified with $2n+1-i$. We denote the unique duality on $\Delta^n\star(\Delta^n)^\op$ by 
	\[
	\func{\tau_n: \Delta^n\star(\Delta^n)^\op\to (\Delta^n)^\op \star \Delta^n; 
		i\mapsto \overline{i}.} 
	\]
	When $n$ is clear from context, we will simply denote $\tau_n$ by $\tau$. 
\end{notation}

\begin{definition}\label{defn:CosimpQ}
	We define a cosimplicial object 
	\[
	\func{
		Q: \Delta^\op \to \scsSet; 
		[n] \mapsto \Delta^n\star(\Delta^n)^\op
	}
	\]
	by declaring a non-degenerate 2-simplex $\func{\sigma:\Delta^2\to\Delta^n\star(\Delta^n)^\op}$ to be thin if: 
	\begin{itemize}
		\item $\sigma$ factors through $\Delta^n\subset Q(n)$;
		\item $\sigma$ factors through $(\Delta^n)^\op\subset Q(n)$;
		\item $\sigma=\Delta^{\{i,j,\overline{k}\}}$, where $i<j\leq k$; or
		\item $\sigma=\Delta^{\{k,\overline{j},\overline{i}\}}$, where $i<j\leq k$.
	\end{itemize}
	Note that the scaling is symmetric under $\tau_n$ by definition; i.e. the maps $\tau_n$ define dualities on the scaled simplicial sets $Q([n])$. 
	
	The `nerve' operation associated to $Q$ is a functor 
	\[
	\func*{Z^\ast: \scsSet\to \Set_{\Delta}
	}
	\]  
	defined by setting $(Z^\ast X)_n:=\Hom_{\scsSet}(Z([n]),X)$. 
\end{definition}

\begin{remark}
	We will often abuse notation and denote $Q([n])$ by $Q(n)$. We will adopt a similar convention for other cosimplicial objects without comment. 
\end{remark}

\begin{definition} 
	Let $\CC$ be an $\infty$-bicategory with underlying $\infty$-category $\C$. The \emph{enhanced twisted arrow category} of $\CC$ is the marked simplicial set 
		\[
		\Tw(\CC):= (Q^\ast\CC, E)
		\] 
	where the edges of $E$ are precisely those corresponding to maps $\func{\Delta^3_\sharp\to \CC}$. Note that the inclusions $\Delta^n_\sharp\subset Q(n)$ and $(\Delta^n)^\op_\sharp\subset Q(n)$ induce a canonical map 
	\[
	\Tw(\CC)\to \C\times \C^\op
	\]
	of simplicial sets. 
\end{definition}

\begin{remark}
	It is immediate from the definitions that $\Tw(\C)$ is the $\infty$-categorical twisted arrow category of \cite{LurieDAGX}. With some work it can be shown that this is precisely the simplicial subset of $\Tw(\CC)$ spanned by the marked morphisms. 
\end{remark}

The immediate aim of this section is to prove the following theorem, which can be seen as an $(\infty,2)$-categorical analogue of \cite[Prop. 4.2.3]{LurieDAGX}. 

\begin{theorem}\label{thm:TwFibrant}
	For any $\infty$-bicategory $\CC$ with underlying $\infty$-category $\C$, the canonical map 
	\[
	\func{\Tw(\CC)\to \C\times \C^\op}
	\]
	 is a Cartesian fibration and the marked edges are Cartesian. 
\end{theorem} 

The proof of \autoref{thm:TwFibrant}, while it involves some combinatorial yoga, begins with the usual, straightforward approach: for each $0< i\leq n$, we consider the lifting problems 
\begin{equation}\label{diag:TwFiblift}
\begin{tikzcd}
(\Lambda^n_i)^\flat \arrow[d]\arrow[r] & \Tw(\CC)\arrow[d]\\
(\Delta^n)^\flat\arrow[r]\arrow[ur,dashed] & \C\times \C^\op
\end{tikzcd}
\end{equation}
and pass to adjoint lifting problems. It is worth noting that, in the case $i=n$, we will in fact consider the edge $\Delta^{\{n-1,n\}}\subset \Lambda^n_n$ to be marked.

\begin{construction}
	The adjoint lifting problem to (\ref{diag:TwFiblift}) will be the extension problem
	\begin{equation}\label{diag:adjointTwFib}
	\begin{tikzcd}
	(K^n_i)_\dagger\arrow[d]\arrow[r] & \CC\\
	Q(n)\arrow[ur,dashed]
	\end{tikzcd}
	\end{equation}
	where $(K^n_i)_\dagger\subset Q(n)$ is the scaled simplicial subset consisting of those simplices $\sigma:\Delta^m\to Q(n)$ which fulfill one of the following three conditions.
	\begin{itemize}
		\item $\sigma$ factors through $\Delta^n\subset Q(n)$.
		\item $\sigma$ factors through $(\Delta^n)^\op\subset Q(n)$.
		\item There exists an integer $j\neq i$ such that neither $j$ nor $\overline{j}$ is a vertex of $\sigma$. 
	\end{itemize}
\end{construction}

\begin{construction}
	We denote by $Q(n)_{\diamond}$ the scaled simplicial set defined by adding to the scaling of $Q(n)$ the triangles of the form $\set{n-1,n,\overline{j}}, \set{n-1,\overline{n},\overline{j}}$ as well as their duals induced by $\tau$. It is immediate to observe that solutions to the lifting problem
	\[
		\begin{tikzcd}[ampersand replacement=\&]
			(K^n_n)_{\diamond} \arrow[r] \arrow[d] \& \CC  \\
			Q(n)_{\diamond} \arrow[ur,dotted]
		\end{tikzcd}
	\]
	correspond to solutions to (\ref{diag:TwFiblift}) with $i=n$, mapping the last edge in $\Lambda^n_n$ to a marked edge in $\tw(\CC)$.
\end{construction}

\begin{construction}
	Let $0<i \leq n$ and define $\scr{K}^{n}_i$ to be the simplicial set obtained by adding to $K^n_i$ the faces $d^{0}$ and $d^{2n+1}$. Denote by $(\scr{K}^{n}_i)_{\dagger}$, $(\scr{K}^n_n)_{\diamond}$ the resulting simplicial sets obtained via the induced scaling.
\end{construction}

Our proof will proceed by showing that both morphisms in each factorization 
\[
\func{
	(K^n_i)_\dagger \to (\scr{K}^n_i)_\dagger \to Q(n)
}
\]
and 
\[
\func{
	(K^n_n)_\diamond \to (\scr{K}^n_n)_\diamond \to Q(n)_\diamond 
}
\]
are scaled anodyne. 

\begin{lemma}\label{lem:fibstep1}\
	\begin{enumerate}
		\item For $0<i<n$ the morphism $\func{(\scr{K}^n_i)_\dagger \to Q(n)}$ is scaled anodyne. 
		\item The morphism $\func{(\scr{K}^n_n)_\diamond \to Q(n)_\diamond }$ is scaled anodyne.
	\end{enumerate}
\end{lemma}

\begin{proof}
	For $0<i\leq n$, we note that unwinding the definition shows that $\scr{K}^n_i=\scr{S}^{\mathcal{A}_i}$, where $\scr{A}_i\subset \mathbf{P}(2n+1)$ is the dull subset containing $\{0\}$, $\{2n+1\}$, and $\{j,\overline{j}\}$ for $0<j\leq n$ such that $j\neq i$. The lemma follows immediately from \autoref{lem:pivot}. 
\end{proof}

\begin{lemma}\label{lem:fibstep2}\
	\begin{enumerate}
		\item For $0<i<n$ the morphism $\func{(K^n_i)_{\dagger} \to (\scr{K}^n_i)_{\dagger}}$ is scaled anodyne.
		\item For $i=n$ the morphism $\func{(K^n_n)_{\diamond} \to (\scr{K}^n_i)_{\diamond}}$ is scaled anodyne.
	\end{enumerate}
\end{lemma}
\begin{proof}
	Let $0<i\leq n$ and note that since $d^{0}\cap d^{2n+1}\in K^n_i$ it will suffice to show that the top horizontal morphism,
	\[
	\begin{tikzcd}[ampersand replacement=\&]
	Y^{\epsilon}_i \arrow[r] \arrow[d] \arrow[dr,phantom,"\lrcorner",very near start] \& \Delta^{2n} \arrow[d,"d^{\epsilon}"] \\
	(K^{n}_i)_{\ast} \arrow[r] \& (\Delta^{2n+1})_{\ast}
	\end{tikzcd}
	\]
	where $\epsilon \in \set{0,2n+1}$ and $\ast\in \set{\dagger,\diamond}$, is scaled anodyne. 
	
	We will first deal with the case $\epsilon=0$. Let $1\leq r \leq n$ and define
	\[
	\func{\sigma_{r}:\Delta^{[r,2n+1]} \to \Delta^{2n+1}}
	\]
	to be the obvious inclusion. Let us remark that $\sigma_1=d^{0}$ and that $\sigma_r$ factors through $d^{0}$ for every possible $r$. We produce a filtration
	\[
	\func{Y^0_i=X_{n+1} \to X_{n} \to  \cdots \to X_{2} \to \Delta^{[1,2n+1]}=X_{1}}
	\]
	where $X_{r}$ is obtained by adding the simplex $\sigma_{r}$ to $X_{r}$.
	
	It will thus suffice for us to check that the upper horizontal morphism in the pullback diagram (i.e. the restriction of $\sigma_r$ to $X_{r+1}$) 
	\[
	\begin{tikzcd}
	Z_r\arrow[r]\arrow[d]\arrow[dr,phantom,"\lrcorner",very near start] & \Delta^{[r,2n+1]}\arrow[d]\\
	X_{r+1}\arrow[r] & \Delta^{[1,2n+1]}
	\end{tikzcd}
	\]
	is scaled anodyne. However, we can observe that $Z_r$ consists of a union in $\Delta^{[r,2n+1]}$
	\begin{itemize}
		\item The $(2n-r)$-dimensional face $d^{r}$.
		\item The $(2n-1)$-dimensional faces $d^{2n+1-j}$ where $0\leq j < r$ and $j \neq i$.
		\item The $(2n-r-1)$-dimensional faces given given by those simplices missing a pair of vertices $\set{j,2n+1-j}$ with $r\leq j \leq n$ and $j \neq i$.
	\end{itemize}
	That is, $Z_r=\scr{S}^{\mathcal{A}_r}$, where $\mathcal{A}_r\subsetneq \mathbf{P}(2n+1-r)$ is the dull subset containing 
	\begin{itemize}
		\item $\{0\}$.
		\item The singletons $\{j\}$ for $2n+1-2r< j\leq 2n+1-r$ with $j\neq 2n-r-i+1$.
		\item The sets $\{k,2n+1-2r-k\}$ for $0\leq k\leq n-r$ with $r+k\neq i$.
	\end{itemize}
	One can easily verify that the scaling satisfies the conditions of \autoref{lem:pivot}, and thus $\func{Z_r\to \Delta^{[r,2n+1]}}$ is scaled anodyne. Consequently, each step of the filtration  
	\[
	\func{Y^0_i=X_{n+1} \to X_{n} \to  \cdots \to X_{2} \to \Delta^{[1,2n+1]}=X_{1}}
	\]
	is scaled anodyne, completing the proof that $\func{Y^0_i\to \Delta^{[1,2n+1]}}$ is scaled anodyne. 
	
	We conclude the proof by noting that the case $\func{Y^{2n+1}_i\to \Delta^{[0,2n]}}$ is formally dual, so by \autoref{rmk:scaledanodyneduals}, the proof is complete. 
\end{proof}

\begin{proof}[Proof of \autoref{thm:TwFibrant}]
	Combining \autoref{lem:fibstep1} and \autoref{lem:fibstep2}, it is immediate that (1) $\Tw(\CC)\to \C\times\C^\op$ is an inner fibration, and (2), the marked edges are Cartesian. It remains only for us to show that there is a sufficient supply of marked edges. Unwinding the definitions, we find that this will be true so long as the lifting problems 
	\[
	\begin{tikzcd}
	\on{Sp^3}\arrow[d]\arrow[r] & \CC\\
	\Delta^3_\sharp\arrow[ur,dashed] 
	\end{tikzcd}
	\]
	admit solutions, where $\on{Sp}_3:=\Delta^{\{0,1\}}\coprod_{\Delta^{\{1\}}}\Delta^{\{1,2\}}\coprod_{\Delta^{\{2\}}}\Delta^{\{2,3\}}$ is the spine of the 3-simplex. However, the left-most map is clearly scaled anodyne so a solution to this problem is guaranteed by fibrancy. The result now follows.
\end{proof} 

\section{The functor classified by $\Tw(\CC)$}\label{sec:classification}
Having now established the Cartesian fibrancy of $\func{\mathbb{T}\!\!\on{w}(\CC)\to \C}$, we aim to determine the functor which it classifies. It will come as no surprise to those familiar with other twisted-arrow category constructions that the functor in question will be the \emph{enhanced mapping functor} of \cite{GHN}, i.e., the mapping category functor of $\CC$ restricted to $\C^\op\times\C$. The solution to this classification problem will be quite involved and technical, involving a number of intermediate $\infty$-categories. Where possible, we will attempt to elucidate the meaning and function of these constructions in the text.  

\subsection{The comparison map}
We now turn our attention to the first step in our proof: constructing the comparison map. This part of the proof will be quite straightforward and in total analogy with its $\infty$-categorical counterpart in \cite{LurieDAGX}. To construct the desired map, we fix, once an for all, the following data:
\begin{itemize}
	\item An $\infty$-bicategory $\CC$ together with its underlying $\infty$-category $\C$. 
	\item A fibrant $\Set_\Delta^+$ enriched category $\DD$, and its maximally marked subcategory (Kan-complex enriched) $\D$, with a commutative diagram 
	\[
	\begin{tikzcd}
	\mathfrak{C}^{\on{sc}}[\C]\arrow[r,"\simeq"]\arrow[d,hookrightarrow] & \D\arrow[d,hookrightarrow]\\
	\mathfrak{C}^{\on{sc}}[\CC]\arrow[r,"\simeq"'] & \DD 
	\end{tikzcd}
	\]
	such that the horizontal arrows are weak equivalences of $\Set_\Delta^+$-enriched categories. 
\end{itemize}
With this data fixed, the enhanced mapping functor is the composite 
\[
\func{F:\D^\op\times \D\to \DD^\op\times\DD\to[{\text{Map}}] \CC\!\on{at}_\infty}
\]
To retain concision, we use the pedestrian notation $F$ for the enhanced mapping functor, rather than the more suggestive $\on{Map}_{\DD}$. 

\begin{proposition}\label{prop:comparisonbeta}
	There is an map 
	\[
	\beta:\Tw(\CC)\to \on{Un}^+_{\C\times\C^\op}(F)
	\]
	of Cartesian fibrations over $\C\times\C^\op$. 
\end{proposition}

\begin{proof}
	The proof proceeds along the same lines as the analogous argument in \cite{LurieDAGX}.  We define an ancillary simplicial category $\EuScript{E}$ with objects either the objects of $\D^\op\times\D$, or a "cone point" $v$. The mapping spaces will be those of $\D^\op\times\D$ if they don't involve $v$, and will be defined by 
	\begin{align*}
	\on{Map}_{\EuScript{E}}(v,(D,D^\prime))&:=\emptyset\\
	\on{Map}_{\EuScript{E}}((D,D^\prime),v) &:= \on{Map}_{\DD}(D,D^\prime)
	\end{align*}
	otherwise. 
	
	As in \cite{LurieDAGX}, a  map  over $\C\times\C^\op$ preserving markings --- $\func{\beta: \Tw(\CC)\to  \on{Un}_{\C\times\C^\op}(F)}$ --- will be equivalent to giving a map 
	\[
	\func{\gamma:\Tw(\CC)^{\triangleright}\to \Nerv(\EuScript{E})}
	\]
	such that the diagram 
	\[
	\begin{tikzcd}
	\operatorname{Tw}(\mathbb{C})^{\triangleright} \arrow[rrrd, dashed, shift left] &                                         &                                                                             &                     \\
	\operatorname{Tw}(\mathbb{C}) \arrow[r] \arrow[u]                               & \C \times \C^{\operatorname{op}} \arrow[r] & \operatorname{N}(\D)\times \operatorname{N}(\D)^{\operatorname{op}} \arrow[r] & \operatorname{N}(\scr{E})
	\end{tikzcd}
	\]
	commutes, and such that, for every $f:\Delta^1\to \Tw(\CC)$ which is marked, the two-simplex $f\ast \id_{\Delta^0}:\Delta^1\star\Delta^0\to \Tw(\CC)^{\triangleright}\to \Nerv(\EuScript{E})$ is sent to a scaled 2-simplex in $\Nsc(\EuScript{E})$.\footnote{Note that (1) this follows directly from unwinding the characterization of the marked 1-simplices of the straightening in  \cite[3.2.1.2]{HTT}, and (2) this is in nearly precise analogy with the definition of the scaling on the scaled cone in \cite[3.5.1.]{LurieGoodwillie}}  (Equivalently, if the adjoint map $\mathfrak{C}[\Delta^2]\to \EuScript{E}$ determines a marked edge in the mapping space from $0$ to $2$.)
	
	We now define the map in question: Given an $n$-simplex $\sigma:\Delta^n\to \Tw(\CC)$, we obtain by definition and adjunction a map 
	\[
	\func{\nu_\sigma:\mathfrak{C}[\Delta^{2n+1}]\to \mathfrak{C}[\CC]\to \DD.} 	
	\]
	We now define a map
	\[
	\func{\gamma_\sigma:\mathfrak{C}[\Delta^{n+1}]\to \EuScript{E}}
	\]
	On $\mathfrak{C}[\Delta^n]\subset \mathfrak{C}[\Delta^{n+1}]$, this is completely determined by the commutativity condition above. For mapping spaces involving the $(n+1)$\textsuperscript{st}-vertex, we define the maps
	\[
	\func*{
		\zeta_i:\OO^{n+1}(i,n+1) \to \OO^{2n+1}(i,2n+1-i);
		S\cup\{n+1\} \mapsto S\cup \tau(S) 
	}
	\]
	where $S$ is considered as a subset of $[n]$, and $\tau$ is the involution on vertices of $\Delta^{2n+1}$. We then define 
	\[
	\func{\gamma_\sigma: \OO^{n+1}(i,n+1)\to[\zeta] \OO^{2n+1}(i,2n+1-i)\to[\nu_\sigma] \on{Map}_{\DD}(\nu_{\sigma}(i), \nu_\sigma(2n+1-i))}
	\]
	Completing our definition of $\gamma_\sigma$, and thus of $\gamma$. It is obvious from our definitions that $\gamma$ respects the marking/scalings.
\end{proof}

\begin{remark}
	The definition of the maps $\zeta_i$ which allow us to define the map above are quite ad-hoc in appearance, as indeed are their analogues in \cite{LurieDAGX}. Once we pass to fibers, the map can be much more elegantly defined: in terms of a composite with a map of posets (see \autoref{rmk:zetasbecomeposetmap}).
\end{remark}

The goal of the remainder of this section will be the proof of the following. 

\begin{theorem}\label{thm:TwClassifiesMap}
	The map $\beta$ is an equivalence of Cartesian fibrations over $\C\times\C^\op$.
\end{theorem}

\subsection{Interlude: a compendium of cosimplicial objects} 

In the sections which follow, there will be a variety of cosimplicial objects in play, each relating to a specific construction necessary for the proof. For ease of reference, we list these here, and describe additional structures (in particular ordered partitions) which will come into play in their study. 

\begin{definition}[The compendium]\label{defn:cosimplicialcompendium}
	We fix, for the rest of the section, the following cosimplicial objects, along with ordered partitions of their $n$\textsuperscript{th} levels. 
	\begin{enumerate}
		\item A cosimplicial object
		\[
		\func*{\bigstar:\Delta\to \scsSet; 
		[n]\mapsto \Delta^n\star \Delta^0}
		\]
		where the scaling on $\bigstar(n)=\Delta^n\star \Delta^0$ is given by declaring every 2-simplex in $\Delta^n\subset \bigstar(n)$ to be thin. 
		\begin{itemize}
			\item We define an ordered partition $(J_0^{\bigstar},J_1^{\bigstar})$ of $\bigstar(n)$ for each $n$ by setting $J_0^{\bigstar}=[n]$ and $J_1^{\bigstar}= \{n+1\}$, where we have identified $\bigstar(n)$ with $\Nerv([n+1])$. 
		\end{itemize}
		\item A cosimplicial object 
		\[
		\func*{
			\boxtimes: \Delta\to \scsSet;
			[n] \mapsto \Delta^n\star (\Delta^n)^\op\star \Delta^0	
		}
		\]
		We use our existing notational conventions for objects of $Q(n)\subset \boxtimes(n)$, and denote the final vertex by $v$. We equip $\boxtimes(n)$ with a scaling by (1) requiring the inclusions $Q(n)\subset \boxtimes(n)$ and $((\Delta^n)^\op\star\Delta^0)_\sharp \subset \boxtimes(n)$ to be maps of scaled simplicial sets; and (2) declaring any simplex of the form $\Delta^{\{j,\overline{i},v\}}$, where $0\leq i \leq j\leq n$, to be thin. 
		\begin{itemize}
			\item We define an ordered partition $(J_0^\boxtimes,J_1^\boxtimes)$ of $\boxtimes(n)$ for each $n$ by setting $J_0^\boxtimes=[n]$ and $J_1^\boxtimes=\{n+1,n+2,\ldots 2n+2\}$ under the identification of $\boxtimes(n)$ with $\Nerv([2n+2])$. 
		\end{itemize}
		\item A cosimplicial object 
		\[
		\func*{\Box:\Delta \to \scsSet; 
			[n] \mapsto \Delta^n\times \Delta^1
		}
		\]
		where the scaling consists of those triangles factoring through $\Delta^n \times \Delta^{\set{1}}$ and those specified in \cite[4.1.5]{LurieGoodwillie}.
		\begin{itemize}
			\item We define an ordered partition $(J_0^\Box,J_1^\Box)$ of $\Box(n)$ by setting $J_0^\Box:=[n]\times\{0\}$ and $J_1=[n]\times \{1\}$ under the identification of $\Box$ with $\Nerv([n]\times[1])$.
		\end{itemize}
	\end{enumerate}
\end{definition}

\subsection{Comparison with outer Cartesian slices}

Having established the existence of a comparison map $\func{\beta:\Tw(\CC)\to F}$ of Cartesian fibrations, we now must pause and circumnavigate our way to a proof that it is an equivalence. The winding route we take will make use of a Cartesian fibration $\CC_{/y}$ defined in \cite{HarpazEquivModels}. The utility  of $\CC_{/y}$ for us lies in the fact that, as established in \cite[\S 2.3]{HarpazEquivModels}, $\CC_{/y}$ classifies the contravariant Yoneda embedding $\scr{Y}_y$ on $\DD$.  In spite of the fact that $\func{\CC_{/y}\to \C}$ is a Cartesian fibration, we will refer to $\CC_{/y}$ as the \emph{outer Cartesian slice category}, in recognition of the fact that our $\func{\CC_{/y}\to \C}$ is a pullback along the inclusion $\func{\C\to \CC}$ of an outer Cartesian fibration as defined in \cite{HarpazEquivModels}. We begin by recalling the definition of $\CC_{/y}$. 

\begin{definition}
	For an object $y\in \CC$, we define the \emph{outer Cartesian slice category} $\CC_{/y}$ whose $n$-simplices are given by maps
	\[
	   \func{\sigma: \bigstar(n) \to \CC}, \text{ such that } \restr{\sigma}{n+1}=y.
	\]
	We equip $\CC_{/y}$ with a marking by declaring an edge to be marked precisely when it can be represented by a map $\func{\Delta^2_\sharp\to \CC}$. Note that the canonical inclusion $\Delta^n_\sharp \subset \bigstar(n)$ induces a map $\func{\CC_{/y}\to \C}$.   
\end{definition}

\begin{proposition}[{\cite[Cor. 2.18]{HarpazEquivModels}}]
	The functor $\func{\CC_{/y}\to \C}$
	is a Cartesian fibration, and an edge of $\CC_{/y}$ is Cartesian if and only if it is marked. 
\end{proposition}

Since our mode of proof is so circuitous, let us take a moment to sketch the path we will take. We begin by showing that there is a span 
\[
\begin{tikzcd}
(\CC_{/y})_x  & \scr{M}_{x,y}\arrow[r,"\sim"]\arrow[l,"\sim"'] & \Tw(\CC)_{(x,y)}
\end{tikzcd}
\]
displaying a weak equivalence of the fibers of $\CC_{/y}$ and $\Tw(\CC)$. 

We then show that there is a weak equivalence  
\[
\func{\on{Un}^+_{\ast}(\mathfrak{C}^{\on{sc}}[\DD](x,y))\to  \on{Un}^{\on{sc}}_{\ast}(\mathfrak{C}^{\on{sc}}[\DD](x,y))}.
\]
From \cite{HarpazEquivModels}, there is an equivalence 
\[
\func{f: (\CC_{/y})_x \to \on{Un}^{\on{sc}}_{\ast}(\mathfrak{C}^{\on{sc}}[\DD](x,y))}.
\] 
The final step to showing that $\beta$ is an equivalence is therefore establishing that the diagram 
\[
\begin{tikzcd}
(\CC_{/y})_x\arrow[d,"f","\simeq"']  & \scr{M}_{x,y}\arrow[r,"\sim"]\arrow[l,"\sim"'] & \Tw(\CC)_{(x,y)}\arrow[d, "\beta"]\\
\on{Un}^{\on{sc}}_{\ast}(\mathfrak{C}^{\on{sc}}[\DD](x,y)) & & \on{Un}^+_{\ast}(\mathfrak{C}^{\on{sc}}[\DD](x,y))\arrow[ll,"\simeq"]
\end{tikzcd}
\]
commutes up to equivalence. 

We begin this journey in the present section by defining the span
\[
\begin{tikzcd}
(\CC_{/y})_x  & \scr{M}_{x,y}\arrow[r,"\sim"]\arrow[l,"\sim"'] & \Tw(\CC)_{(x,y)}
\end{tikzcd}
\]
and showing that its legs are weak equivalences. 

\begin{notation}
	Let $x,y\in \CC$. We denote by $\Tw(\CC)_y$ the pullback 
	\[
	\begin{tikzcd}
	\Tw(\CC)_y\arrow[r]\arrow[d] & \Tw(\CC)\arrow[d]\\
	\C \arrow[r,"\id\times\{y\}"'] & \C\times \C^{\op} 
	\end{tikzcd}
	\] 
	and by $\Tw(\CC)_{(x,y)}$ the fiber over $(x,y)\in \C\times\C^\op$. 
\end{notation} 

\begin{definition}
	For $y\in \CC$, we define a simplicial set $\scr{M}_y$ whose $n$-simplices are given by maps
	\[
		\func{\sigma:\boxtimes\!(n) \to \CC} \text{ such that } \enspace \restr{\sigma}{\Nerv\left(J_1^{\boxtimes}\right)}=y.
	\]
	Note that these are, equivalently, maps $\boxtimes(n)^R \to \CC$. The inclusion $\Delta^n_\sharp=\Nerv \left(J_0^\boxtimes\right)_\sharp\subset \boxtimes(n)$ induces a map $\func{\scr{M}_y\to \C}$. 
\end{definition}

\begin{remark}
	We will view $\bigstar(n), \, \boxtimes(n)$, and $Q(n)$ as equipped with their ordered partitions from \autoref{defn:cosimplicialcompendium} and \autoref{ex:orderedpartofQ}, and consider their right quotients $\bigstar(n)^R$, $\boxtimes(n)^R$, and $Q(n)^R$, each of which piece together to form a cosimplicial object in $\scsSet$. The obvious natural inclusions 
	\[
	\begin{tikzcd}
	{\bigstar^R} \arrow[r,Rightarrow] & {\boxtimes^R} & Q^R \arrow[l,Rightarrow]
	\end{tikzcd}
	\]
	then induce maps 
	\[
	\begin{tikzcd}
		\CC_{/y}  & \scr{M}_{y}\arrow[r,"\pi"]\arrow[l,"\rho"'] & \Tw(\CC)_{y}
	\end{tikzcd}
	\]
	over $\C$. 
\end{remark}

\begin{proposition}\label{prop:PiTrivKanFib}
	The map $\func{\pi:\scr{M}_y\to \Tw(\CC)_y}$ is a trivial Kan fibration. 
\end{proposition}

\begin{proof}
	We first aim to show that the inclusions $\func{i_n:Q(n)^R\to \boxtimes(n)^R}$ are scaled trivial cofibrations. To this end, we define a map 
	\[
	\func*{r_n: \boxtimes(n)\to Q(n); 
	i\mapsto \begin{cases}
	i & i<2n+2\\
	i-1 & i=2n+2 
	\end{cases}
	}
	\]
	We see immediately that $r_n$ descends to a map $\func{r_n:\boxtimes(n)^R\to Q(n)^R}$, and that $r_n\circ i_n=\id$. Moreover, one can check that the natural transformation $i_n\circ r_n\Rightarrow \id$ descends to a transformation
	\[
	\func{\Delta^1_\flat \times \boxtimes(n)^R\to \boxtimes(n)^R}
	\]
	whose components are degenerate. Consequently, we see that $i_n$ is an equivalence of scaled simplicial sets. We then consider the boundary lifting problem and its associated adjoint problem
	\[
	\begin{tikzcd}
	\partial\Delta^n\arrow[r]\arrow[d] & \scr{M}_y\arrow[d,"\pi"]\\
	\Delta^n\arrow[r]\arrow[ur,dashed] & \Tw(\CC)_y
	\end{tikzcd}
	\enspace \enspace \rightsquigarrow
	\begin{tikzcd}
	K^n\arrow[d]\arrow[r] & \CC \\ 
	\boxtimes(n)^R\arrow[ur,dashed] 
	\end{tikzcd}
	\enspace \enspace  K^n=\partial( \boxtimes^R)^{n}\coprod\limits_{\partial (Q^R)^n}Q(n)^R
	\]
	Examining \autoref{defn:orderedpartitions}, we note that we can extend our conventions to $Q(\emptyset)^R=\Delta^0$ and $\boxtimes(\emptyset)^R=\Delta^0$. Consequently, we can write 
	\[
	\partial(\boxtimes^R)^n=\hocolim\limits_{I\subsetneq [n]} \boxtimes(I)^R \quad \text{and} \quad 	\partial (Q^R)^n=\hocolim\limits_{I\subsetneq [n]} Q(I)^R. 
	\]  
	Since the two diagrams are naturally equivalent, this yields an equivalence $\func{ \partial (Q^R)^n \to[\simeq] \partial(\boxtimes^R)^n}$. Since this map is a trivial cofibration it follows that $Q(n)^R \to K^n$ is an equivalence. Finally we consider the factorization
	\[
		\func{Q(n)^R \to K^n \to \boxtimes(n)^R }
	\]
	and we conclude by 2-out-of-3 that the map $K^n \to  \boxtimes(n)^R$ is a trivial cofibration. This finishes the proof.
\end{proof}

\begin{corollary}
	The map $\scr{M}_y\to \C$ is a Cartesian fibration, and $\func{\pi:\scr{M}_y\to \Tw(\CC)_y}$ is an equivalence of Cartesian fibrations over $\C$. 
\end{corollary}

\begin{lemma}
	The Cartesian edges of $\scr{M}_y$ over $\C$ are precisely those which can be represented by scaled maps $\boxtimes^1_\dagger\to \CC$, where $\dagger$ is the extension of the scaling on $\boxtimes^1$ to include 
	\begin{enumerate}
		\item all 2-simplices in $\Delta^1\star(\Delta^1)^\op$, and 
		\item the 2-simplex $\Delta^{\{01,v\}}$. 
	\end{enumerate}
\end{lemma}

\begin{proof}
	Left as an exercise to the reader. 
\end{proof}

\begin{corollary}
	The map $\func{\rho:\scr{M}_y\to \CC_{/y}}$ is a map of naturally-marked Cartesian fibrations over $\C$. 
\end{corollary}

\begin{proposition}
	For any $x\in \C$, denote the fiber of $\scr{M}_y$ over $x$ by $\scr{M}_{x,y}$. Then the induced map 
	\[
	\func{\rho:\scr{M}_{x,y}\to (\CC_{/y})_x}
	\]
	is a trivial Kan fibration. 
\end{proposition}

\begin{proof}
	We follow effectively the same method as in the proof of \autoref{prop:PiTrivKanFib}, now using the two-sided quotients $\widetilde{\bigstar}(n)$ and $\widetilde{\boxtimes}(n)$ of the defining cosimplicial objects. By the same homotopy colimit argument, it will suffice for us to show that 
	\[
	\func{i_n:\widetilde{\bigstar}(n)\to \widetilde{\boxtimes}(n)}
	\] 
	is an equivalence. However, $i_n$ is already a bijection on objects, so it will suffice for us to show that $i_n$ induces an equivalence on the single non-trivial mapping space. To this end, we make use of the characterization of \autoref{lem:posetquotientmappingspaces}. It will thus suffice to to show that the maps of marked simplicial sets
	\[
	\func{\widetilde{s}:
	\left(\scr{P}_{\bigstar(n)}\right)_{/\sim_A}\to \left(\scr{P}_{\boxtimes(n)}\right)_{/\sim_A}
	}
	\]
	are equivalences for any $n$. For the rest of the proof we will abuse notation and denote the nerves of these posets by $\scr{P}_{\bigstar}$ and $\scr{P}_{\boxtimes}$.
	
	We will work with the unquotiented simplicial sets, and define maps which descend to quotients. Before we can do this, however, we must fix some notation. We denote object $S\in \scr{P}_{\boxtimes}$ by triples $(S_0,S_1,S_2)$ of subsets of each of the three joined components in $\boxtimes(n)=\Delta^n\star (\Delta^n)^\op\star \Delta^0$. We will similarly denote objects of $P_{\bigstar}$ by pairs $(S_0,v)$ of sets. Note that with this new coordinates the unquotient version of $\widetilde{s}$ can be described as
	\[
		\func{s:\scr{P}_{\bigstar}\to \scr{P}_{\boxtimes};(S_0,v) \mapsto (S_0,\emptyset,v) }
	\] 
	We define $\scr{P}_{G}\subset \scr{P}_{\boxtimes}$ as the nerve of the full subposet on those objects of the form
	\begin{itemize}
		\item $(S_0,\emptyset,v)$.
		\item $(S_0,S_1,v)$ such that	
		\begin{itemize}
			\item  $S_1 \neq \emptyset$ 
			\item  $S_1 \cup \set{v}$ contains all elements of $\boxtimes(n)$ greater than $\min(S_1)$, and 
			\item $\tau(S_1) \subset S_0$.
		\end{itemize}
	\end{itemize}
	We equip $\scr{P}_G$ with the induced marking producing a factorization $\func{\scr{P}_{\bigstar} \to[s_\alpha] \scr{P}_G \to[s_\beta] \scr{P}_{\boxtimes}}$. In light of this fact, we will turn our efforts into showing that $s_\alpha,s_\beta$ descend to equivalences $\widetilde{s}_{\alpha}$ and $\widetilde{s}_{\beta}$. We define a marking-preserving map of posets
	\[
		\func{r_{\alpha}: \scr{P}_G \to \scr{P}_{\bigstar},(S_0,S_1,v) \mapsto (S_0,v)}
	\]
	such that $r_{\alpha} \circ s_\alpha=\on{id}$. We, moreover, observe that there is a natural transformation $\epsilon_{\alpha}:s_{\alpha}\circ r_{\alpha}  \Rightarrow  \on{id}$ whose components are marked in $\scr{P}_{G}$. To check that $\epsilon_{\alpha}$ (and consequently $r_{\alpha}$) factors through the quotient it is enough to note that given $k$-simplices $\underline{S} \sim_A \underline{T}$ in $\scr{P}_G$ then it follows that $\underline{S} \sim_L \underline{T}$. We can now conclude that $\widetilde{s}_{\alpha}$ is an equivalence.
	
	We define a map of posets
	\[
	\func*{r_{\beta}: P_{\boxtimes}\to P_G; 
	(S_0,S_1,S_2)\mapsto \begin{cases}
		(S_0,\emptyset,v) \text{ if } S_1=\emptyset \\
		\left(S_0\cup\tau_n([\min(S_1),v)),[\min(S_1),v),v \right) \text{ otherwise.}
	\end{cases}
	}
	\] 
	such that $r_{\beta} \circ s_{\beta}=\on{id} $ and note that there is a map $S\to r_{\beta}(S)$ inducing a natural transformation $\epsilon_{\beta}: \on{id} \Rightarrow s_{\beta}\circ r_{\beta}$. To show that $\func{\scr{P}_G \to \scr{P}_{\boxtimes}}$ is an equivalence, it is sufficient to check that $r_{\beta}$ preserves markings, that $\epsilon_{\beta}$ descends to quotients, and that the components of the natural transformation become equivalences in the fibrant replacement of the localizations. We will prove here that $\epsilon_{\beta}$ descends to the quotient leaving the rest of the checks as exercises for the interested reader. Let $\underline{S}\sim_A \underline{T}$ be $k$-simplices and denote by ${}^\beta \underline{S}$ and ${}^{\beta}\underline{T}$ their images under $r_{\beta}$. Let $s_0^R,s_0^L$ be the truncation points for $\underline{S}$ and denote by ${}^{\beta}\!s_0^R, {}^{\beta}\!s_0^L$ the truncation points for ${}^{\beta}\!\underline{S}$. It is immediate to see that
	\[
	  {}^{\beta}\!s_0^R=s_0^R, \enspace \enspace \enspace {}^{\beta}\!s_0^L=\begin{cases}
	  	\max \big\{ s_0^L {,} \tau({}^{\beta}\!s_0^R)\big\}\enspace  \text{ if } \enspace (S_0)_1 \neq \emptyset \\
		s_0^{L} \text{ otherwise.}
	  \end{cases}
	\]
	 
	 This implies that, in order to show our claim, it suffices to check that for every $\ell \in [k]$ the ambidextrous truncations of ${}^{\beta}\!S_\ell$, ${}^{\beta}T_{\ell}$ with respect to $s_0^L,s_0^R$ coincide. If $(S_\ell)_1\neq \emptyset$ the conclusion follows immediately. We will also assume that $\kappa=\tau(\min((S_l)_1)) < \max((S_l)_0)$, since otherwise we would have ${}^{\beta}\!S_\ell={}^{\beta}T_\ell$. Denote by ${}^{\beta}\!\widehat{S}^A_\ell$ the truncation with respect to our chosen points. Then we observe that
	 \[
	 	{}^{\beta}\!\widehat{S}_\ell^A=\begin{cases}
	 		[s_0^L,\kappa] \cup (S_{\ell})_0^{\geq \kappa} \cup [\tau(\kappa),v] \enspace \text{ if } s_0^L \leq \kappa \\
			(S_{\ell})_0 \cup [\tau(\kappa),v] \text{ otherwise}
	 	\end{cases}
	 \]
	 where $(S_{\ell})_0^{\geq \kappa}$ stands for the obvious notation. Since this only depends on $S_{\ell}^A$ it follows that ${}^{\beta}\!\widehat{S}_\ell^A={}^{\beta}\widehat{T}_\ell^A$.
\end{proof}

We thus have completed the first step of the proof: 

\begin{corollary}\label{cor:harpazsliceisTwFib}
	The maps 
	\[
	\begin{tikzcd}
		\CC_{/y}  & \scr{M}_{y}\arrow[r,"\pi"]\arrow[l,"\rho"'] & \Tw(\CC)_{y}
	\end{tikzcd}
	\]
	are equivalences of naturally marked Cartesian fibrations over $\C$. 
\end{corollary}

\begin{remark}
	This would already be sufficient, in light of \cite[\S 2.3]{HarpazEquivModels}, for us to conclude that $\Tw(\CC)_y$ classifies the restriction to $\C$ of the representable functor defined by $y$. It is not, however, sufficient to show that $\Tw(\CC)$ classifies the enhanced mapping functor. We still have work to do. 
\end{remark}

\subsection{Comparing the comparisons}

By \cite{HarpazEquivModels}, there are equivalences of marked simplicial sets 
\[
\func{(\CC_{/y})_x\to[\sim]  (\CC^{x/})_y \to[\sim] \on{Un}^{\on{sc}}_\ast\left(\mathfrak{C}[\CC](x,y)\right)}
\]
where $\on{Un}^{\on{sc}}_\ast$ is the scaled \emph{coCartesian} unstraightening, and $\CC^{x/}$ denotes the scaled slices defined using the fat join in \cite[4.1.5]{LurieGoodwillie}.\footnote{Technically speaking, in \cite{LurieGoodwillie} Lurie defines a scaled coCartesian fibration $\func{\overline{\CC^{x/}}\to \CC}$. We will make use of the pullback along $\func{\C\to \CC}$, and denote it by $\CC^{x/}$.} We also have, by \autoref{prop:comparisonbeta}, a comparison map
\[
\beta:\Tw(\CC)_{x,y}\to \on{Un}^{+}_\ast\left(\mathfrak{C}[\CC](x,y)\right),
\]  
where $\on{Un}^+_\ast$ is the marked \emph{Cartesian} unstraightening. 

We now aim to compare these two comparison maps, using the equivalence between $\Tw(\CC)_y$ and $\CC_{/y}$ of \autoref{cor:harpazsliceisTwFib}. The first step is to relate the scaled coCartesian and marked Cartesian straightenings over the point. 

\begin{construction}
	 We denote the former by $\on{St}^{\on{sc}}$ and the latter by $\on{St}^+$, leaving the point implicit. These give us two functors 
	\[
	\func*{\on{St}^{\on{sc}},\on{St}^{+}:\Set_{\Delta^+}\to \Set_{\Delta}^+}
	\]
	By \cite[Lem. 4.3.3]{ADS2Nerve}, to display a natural equivalence between them it will suffice to display it on simplices. By definition, we have that
	$\on{St}^{\on{sc}}(\Delta^n) =\mathfrak{C}^{\on{sc}}[\widetilde{\Box}(n)](x,y) 
	$
	and 
	$
	\on{St}^{+}(\Delta^n) =\mathfrak{C}^{\on{sc}}[\widetilde{\bigstar}(n)](x,y). 
	$
	
	Since the collapse map 
	\[
	\func*{\Delta^n\times\Delta^1\to \Delta^n\star\Delta^0;
		(i,k)\mapsto \begin{cases}
		i & k=0 \\
		n+1 & k=1
		\end{cases}}
	\]
	preserves the scaling and ordered partitions, we thus obtain compatible maps 
	$
	\func{\theta_n:\on{St}^{\on{sc}}((\Delta^n)^\flat)\to \on{St}^{+}((\Delta^n)^\flat)}
	$
	and 
	$
	\func{\on{St}^{\on{sc}}((\Delta^1)^\sharp)\to \on{St}^{+}((\Delta^n)^\sharp)}
	$
	Moreover, the triangles 
	\[
	\begin{tikzcd}
	\on{St}^{\on{sc}}((\Delta^n)^\flat)\arrow[rr,"\theta_n"]\arrow[dr,"p"'] & & \on{St}^{+}((\Delta^n)^\flat)\arrow[dl,"q"]\\
	& (\Delta^n)^\flat 
	\end{tikzcd}
	\]
	commute, where $q$ is the map $\pi$ of \cite[Prop 3.2.1.14]{HTT}, and $p$ is the map $\alpha$ of \cite[Prop. 3.6.1]{LurieGoodwillie}. Since both of these are marked equivalences, we have that $\theta_n$ is as well. Thus, $\theta$ extends to a natural equivalence $\func{\theta:\on{St}^{\on{sc}}\nat \on{St}^+}$. 
\end{construction}

It immediately follows that 

\begin{lemma}\label{lem:equivofunst}
	The natural transformation $\mu:\on{Un}^{\on{sc}}\to \on{Un}^+$ adjoint to $\theta$ is an equivalence. 
\end{lemma}

It now remains only for us to show 

\begin{proposition}\label{prop:comparisonscommute}
	The diagram 
	\begin{equation}\label{diag:equivcomparisons}
	\begin{tikzcd}
	(\CC_{/y})_x\arrow[d,"f","\simeq"']  & \scr{M}_{x,y}\arrow[r,"\sim"]\arrow[l,"\sim"'] & \Tw(\CC)_{(x,y)}\arrow[d, "\beta"]\\
	\on{Un}^{\on{sc}}_{\ast}(\mathfrak{C}^{\on{sc}}[\DD](x,y)) & & \on{Un}^+_{\ast}(\mathfrak{C}^{\on{sc}}[\DD](x,y))\arrow[ll,"\simeq"]
	\end{tikzcd}
	\end{equation}
	commutes up to natural equivalence. 
\end{proposition}

To effect a proof, we first note that the maps $f$ and $\beta$ are both induced by maps of posets. For the reader's convenience, we briefly unwind how in the case of $\beta$. For $f$, we merely state the poset map in question, and leave it to the interested reader to unwind the definitions. 

\begin{remark}\label{rmk:zetasbecomeposetmap}
	Given an $n$-simplex $\sigma$ of $\Tw(\CC)_{(x,y)}$, the simplex $\beta(\sigma)$ of $\on{Un}^{+}\left(\mathfrak{C}^{\on{sc}}[\DD](x,y)\right)$ is given by pulling back the rigidification $\mathfrak{C}[\tilde{\sigma}]$ of the adjoint map 
	\[
	\func{\tilde{\sigma}: \widetilde{Q}(n) \to \CC}
	\]
	along the maps $\zeta_i$ constructed in \autoref{prop:comparisonbeta}. Using the poset-quotient description of the mapping spaces, however, one can easily check that the $\zeta_i$'s combine to define a map 
	\[
		\func{B:\scr{P}_{\bigstar(n)} \to \scr{P}_{Q(n)}; (S_0,v) \mapsto (S_0,\tau(S_0))}
	\]
	with associated map on the quotient $\func{\widetilde{B}:\mathfrak{C}^{\on{sc}}[\widetilde{\bigstar}(n)](\ast_0,\ast_1)\to\mathfrak{C}^{\on{sc}}[\widetilde{Q}(n)](\ast_0,\ast_1)}$. That is, $\beta(\sigma)$ is defined by pulling back $\mathfrak{C}[\tilde{\sigma}]$ along $\widetilde{B}$. 
\end{remark}   

More generally, let $\sigma$ be a simplex of $\scr{M}_{x,y}$ and $\func{\tilde{\sigma}:\widetilde{\boxtimes}(n)\to \CC}$ its adjoint. The right-hand composite $\func{\gamma:\scr{M}_{x,y}\to \on{Un}^{\on{sc}}_{\ast}(\mathfrak{C}^{\on{sc}}[\DD](x,y))}$ in  (\ref{diag:equivcomparisons})  is given by pulling $\mathfrak{C}[\tilde{\sigma}]$ back along  a map  $\mathfrak{C}^{\on{sc}}[(\Box(n))^R](\ast_0,\ast_1)\to\mathfrak{C}^{\on{sc}}[\widetilde{\boxtimes}(n)](\ast_0,\ast_1)$  induced by\footnote{In point of fact, unraveling the definitions would lead one to believe that map is induced by $\func{(S_0,S_1) \mapsto (S_0,\tau(S_0),\emptyset)}$, however, both this map and $G$ lead to the same map on quotients, so the distinction is irrelevant.}
\[
\func{G: P_{\Box(n)}\to P_{\boxtimes(n)};
(S_0,S_1) \mapsto (S_0,\tau(S_0),v).
}
\]
The left hand composite $\func{\eta:\scr{M}_{x,y}\to \on{Un}^{\on{sc}}_{\ast}(\mathfrak{C}^{\on{sc}}[\DD](x,y))}$ in  (\ref{diag:equivcomparisons}) is given by pulling $\mathfrak{C}[\tilde{\sigma}]$ back along  a map  $\mathfrak{C}^{\on{sc}}[(\Box(n))^R](\ast_0,\ast_1)\to\mathfrak{C}^{\on{sc}}[\widetilde{\boxtimes}(n)](\ast_0,\ast_1)$ induced by 
\[
\func{
H:P_{\Box(n)}\to P_{\boxtimes(n)};
(S_0,S_1)\mapsto (S_0,\emptyset,\{v\}).
}
\]

With these definitions in place, we can proceed to the final step of our proof.

\begin{proof}[Proof of \autoref{prop:comparisonscommute}]
	We will define an explicit homotopy $\func{(\Delta^1)^\sharp\times\scr{M}_{x,y}\to \on{Un}^{\on{sc}}_{\ast}(\mathfrak{C}^{\on{sc}}[\DD](x,y))}$ between $\gamma$ and $\eta$. 
	
	Given an $n$-simplex $\func{(\rho,\sigma):\Delta^n\to (\Delta^1)^\sharp\times \scr{M}_{x,y}}$, we note that $\rho$ is uniquely specified by $0\leq i\leq n+1$: 
	\[
	(0,0,\ldots, 0,\overbrace{1}^i, 1,\ldots, 1). 
	\] 
	We define, for $S_0\subset [n]$, the subset $S_0^{\geq i}:=\{s\in S_0\mid s\geq i\}$ and then define a map 
	\[
	\func{
		h_\rho: P_{\square(n)}\to P_{\boxtimes(n)}; 
		(S_0,S_1)\mapsto (S_0, \tau(S_0^{\geq i}),v)
	}
	\]
	Note that, when $i=n+1$ (i.e. $\rho$ is constant on $0$) we have that $\tau(S_0^{\geq i})=\emptyset$, so that the map specializes to $H$. Similarly, when $i=0$ (i.e. $\rho$ is constant on $1$) $S_0^{\geq i}=S_0$, so that the map specializes to $G$. 
	
	Let us check that $h_{\rho}$ descend to quotients. Note that the case where $\rho$ is constant on $0$. In order to do so, given a $k$-simplex $\underline{S}$ we compute its ambidextrous truncation in $\scr{P}_{\boxtimes(n)}$. Let $l\in [k]$ and denote $(S_l)_0 \cap [s_0^L,n]=\widehat{S}_{l}$. Then we obtain
	\[
		h_{\rho}(\underline{S})^A_l=\begin{cases}
			(\widehat{S}_l,\tau(\widehat{S}_l)) \text{ if } i \leq s_0^L\\
			\left(\widehat{S}_l,\tau(\widehat{S}_l)\cap [n+1,\tau(i)]\right) \text{ if }s_0^L < i < n+1 \\
			(\widehat{S}_l,\emptyset,v) \text{ if }i=n+1
		\end{cases}
	\]
	since this only depends on the truncation of $\underline{S}$ the claim follows. It is immediate that the maps respect the simplicial identities, so sending a simplex $(\rho,\sigma)\in \Delta^1\times \scr{M}_{x,y}$ to the simplex $h_\rho^\ast(\mathfrak{C}[\tilde{\sigma}])\in \on{Un}^{\on{sc}}(\mathfrak{C}[\CC](x,y))$ defines a homotopy $\func{\eta\nat \gamma}$. 
	
	To see that it is a marked homotopy, consider the $1$-simplex $(0,1)$ in $\Delta^1$, and a degenerate $1$-simplex 
	$\func{\sigma:\boxtimes\!(1)_\sharp \to \CC}$
	in $\scr{M}_{x,y}$. This corresponds to a map 
	$\func{(\scr{P}_{\boxtimes(1)})^\sharp \to \mathfrak{C}[\CC](x,y),}$
	and so pulling back along $h_{\{0,1\}}:\scr{P}_{\square(1)}\to \scr{P}_{\boxtimes(1)}$ yields a map 
	\[
	\func{h_{\{0,1\}}^\ast(\gamma):(\scr{P}_{\square(1)})^\sharp\to \mathfrak{C}[\CC](x,y)},
	\]
	i.e., a marked morphism in $\on{Un}^{\on{sc}}(\mathfrak{C}[\CC](x,y))$. We have thus defined a marked homotopy as desired, and the proof is complete.
\end{proof}

For completeness, we can now give

\begin{proof}[Proof of \autoref{thm:TwClassifiesMap}]
	Using \autoref{prop:comparisonscommute}, the theorem follows immediately by 2-out-of-3 from \autoref{cor:harpazsliceisTwFib}, the equivalence of \cite[Prop. 2.24]{HarpazEquivModels}, and \autoref{lem:equivofunst}. 
\end{proof}

\section{Natural transformations as an end}\label{sec:Nat}
In this section we will denote by $X$ be a maximally scaled simplicial set and by $\DD$ an $\infty$-bicategory that will remain fixed throughout. Given a pair of functors $\func{F,G: X \to \DD}$ we will denote the associated mapping category in $\DD^{X}$ by $\on{Nat}\limits_X(F,G)$. The aim of this section is to show that  $\operatorname{Nat}\limits_X(F,G)$ can be expressed as the limit of the functor
\[
\begin{tikzcd}[ampersand replacement=\&]
N_{(F,G)}:\Tw(X)^{\op}  \arrow[r] \& X^{\op}\times X \arrow[rr,"F^{\op}\times G"] \& \& \DD^{\op}\times \DD \arrow[rr,"\on{Map}_{\DD}(\mathblank{,}\mathblank)"] \& \& \goth{Cat}_{\infty}.
\end{tikzcd}
\]
That is to say, the $\infty$-category of natural transformations can be obtained as an end, in the terminology of \cite{GHN}.
\begin{definition}
	Let $\func{\ell: \D^{X}\times \left(\D^{X}\right)^{\op} \to \on{Fun}}(\Tw(X),\D\times \D^{\op})$ be the functor that maps a simplex of the product $\func{\sigma_1: X \times \Delta^n \to \D}$, $\func{\sigma_2: X \times (\Delta^n)^{\op} \to \D}$ to the composite
	\[
	\begin{tikzcd}[ampersand replacement=\&]
	\Tw(X)\times \Delta^n \arrow[r] \& X \times X^{\op} \times \Delta^n \arrow[r,"(\sigma_1{,}\sigma_2^{\op})"] \& \D\times \D^{\op}.
	\end{tikzcd}
	\]
	We define a marked simplicial set $\mathcal{L}_X$ equipped with  Cartesian fibration to $\D^{X} \times \left(\D^X\right)^{\op}$ via the pullback square
	\[
	\begin{tikzcd}[ampersand replacement=\&]
	\mathcal{L}_{X} \arrow[d]\arrow[dr,phantom,"\lrcorner",very near start] \arrow[r] \& \on{Fun}(\Tw(X)^{\sharp},\Tw(\DD)^{\dagger})  \arrow[d] \\
	\D^{X}\times \left(\D^{X}\right)^{\op} \arrow[r,"\ell"] \& \on{Fun}(\Tw(X),\D\times \D^{\op}).
	\end{tikzcd}
	\]
\end{definition}
\begin{remark}
	Given a pair of functors $F,G$ we see that the fiber $\left(\mathcal{L}_{X}\right)_{(F,G)}$ is given by the category of “ways of completing the commutative diagram''
	\[
	\begin{tikzcd}[ampersand replacement=\&]
	\Tw(X) \arrow[r,dotted] \arrow[d] \& \Tw(\DD) \arrow[d] \\
	X \times X^{\op} \arrow[r,"F \times G^{\op}"] \& \D \times \D^{\op}
	\end{tikzcd}
	\]
	In other words, it is precisely the category of Cartesian sections associated to $N_{(F,G)}$. In particular it follows from \cite[Cor. 3.3.3.2]{HTT} that $\left(\mathcal{L}_{X}\right)_{(F,G)}$ is a model for the limit $\lim\limits_{\on{Tw(X)^{\op}}}N_{(F,G)}$. We will abuse notation and denote the fiber by $\mathcal{L}_{(F,G)}$ when the diagram category is clear from the context.
\end{remark}
Recall the canonical map $\func{\DD^{X}\times X \to \DD}$ and note that since $\Tw(\mathblank)$ preserves limits we can use the tensor-hom adjunction to produce 
\[
\func{u:\Tw\left(\DD^{X}\right) \to \on{Fun}(\Tw(X)^{\sharp},\Tw(\DD)^{\dagger})}
\]
fitting into the commutative diagram 
\[
\begin{tikzcd}[ampersand replacement=\&]
\Tw\left(\DD^X\right) \arrow[d] \arrow[r,"\theta"] \& \on{Fun}(\Tw(X)^{\sharp},\Tw(\DD)^{\dagger})  \arrow[d] \\
\D^{X}\times \left(\D^{X}\right)^{\op} \arrow[r,"\ell"] \& \on{Fun}(\Tw(X),\D\times \D^{\op}).
\end{tikzcd}
\]
This in turn yields a map of Cartesian fibrations $\func{\Theta_X: \Tw(\DD^X) \to \scr{L}_X}$ which we call the  \emph{canonical comparison map}. The rest of this section is devoted to showing that $\Theta_X$ is a fiberwise equivalence for every simplicial set $X$. Our first observation is that both constructions behave contravariantly in the simplicial set $X$ thus producing functors
\[
\func{\Tw\left(\DD^{(\mathblank)}\right), \scr{L}_{(\mathblank)}: \on{Set}_{\Delta}^{\op} \to \on{Set}_{\Delta}^{+}}
\]
equipped with a natural transformation $\func{\Theta:\Tw\left(\DD^{(\mathblank)}\right) \nat \scr{L}_{(\mathblank)} }$. We can now state the main result of the section.
\begin{theorem}\label{thm:nat}
	For every simplicial set $X\in\on{Set}_{\Delta}$, the map $\func{\Theta_X:\Tw\left(\DD^{X}\right) \to \scr{L}_{X} }$ is an equivalence of Cartesian fibrations over $\D^{X}\times \left(\D^{X}\right)^{\op}$.
\end{theorem}
Our proof strategy will consist in reducing the problem to the case $X=\Delta^n$ with $n=0,1$. In order to achieve this we will show that the both functors are homotopically well-behaved.
\begin{proposition}\label{prop:cof}
	Let $\func{\alpha:X \to Y}$ be a cofibration of simplicial sets. Then for every pair of functors $F,G \in \D^{Y} $ the induced maps
	\[
	\func{\Tw\left(\DD^{Y}\right)_{(F,G)} \to \Tw\left(\DD^{X}\right)_{(\alpha^{*}F,\alpha^{*}G)}, } \enspace \func{ \left(\scr{L}_{Y}\right)_{(F,G)} \to  \left(\scr{L}_{X}\right)_{(\alpha^{*}F,\alpha^{*}G)}}
	\]
	are fibrations in the Joyal model structure.
\end{proposition}
\begin{proof}
	Let us observe that due to  \autoref{thm:TwFibrant} and \cite[Cor. 2.4.6.5]{HTT}, to check that the first map is a Joyal fibration it will suffice to solve the lifting problems
	\[
	\begin{tikzcd}[ampersand replacement=\&]
	(\Lambda^n_i )^{\flat}\arrow[r] \arrow[d] \& \Tw(\DD^Y)_{(F,G)} \arrow[d] \\
	(\Delta^n)^{\flat} \arrow[r] \& \Tw(\mathbb{D}^X)_{(\alpha^*F,\alpha^* G)}
	\end{tikzcd}
	\begin{tikzcd}[ampersand replacement=\&]
	(\Delta^0)^{\sharp}\arrow[r] \arrow[d] \& \Tw(\DD^Y)_{(F,G)} \arrow[d] \\
	(\Delta^1)^{\sharp} \arrow[r] \& \Tw(\mathbb{D}^X)_{(\alpha^*F,\alpha^* G)}
	\end{tikzcd}
	\]
	with $n\geq 2$ and $0<i<n$. These lifting problems can be easily seen to be equivalent to their adjoint problems (where we are using the notation of the proof of \autoref{thm:TwFibrant})
	\[
	\begin{tikzcd}[ampersand replacement=\&]
	(K^n_i )_{\dagger}\times Y \coprod\limits_{(K^n_i)_{\dagger} \times X} Q(n)\times X  \arrow[r] \arrow[d] \& \DD  \\
	Q(n)\times Y \arrow[ur,dotted]
	\end{tikzcd}
	\enspace  \enspace \enspace \enspace \enspace
	\begin{tikzcd}[ampersand replacement=\&]
	\on{Sp}^3 \times Y \coprod\limits_{\on{Sp}^3 \times X}\Delta^3_{\sharp} \times X  \arrow[r] \arrow[d] \& \DD  \\
	\Delta^3_{\sharp} \times Y \arrow[ur,dotted]
	\end{tikzcd}
	\]
	which admit solution in virtue of \cite[Proposition 3.1.8]{{LurieGoodwillie}}. The proof for the other functor is almost analogous. First we note that the induced map $\Tw(X)\to \Tw(Y)$ is a cofibration of marked simplicial sets. Let $A^{\diamond} \to B^{\diamond}$ be a marked anodyne morphism, then using \cite[Prop. 3.1.2.3]{HTT} we see that lifting problems of the form
	\[
	\begin{tikzcd}[ampersand replacement=\&]
	\Tw(X)^{\sharp}\times B^{\diamond}  \coprod\limits_{\Tw(X)^{\sharp}\times A^{\diamond} } \Tw(Y)^{\sharp}\times A^{\diamond}  \arrow[d] \arrow[r] \& \Tw(\DD)^{\dagger} \arrow[d] \\
	\Tw(Y)^{\sharp}\times B^{\diamond} \arrow[ur,dotted] \arrow[r] \& \D \times \D^{\op}
	\end{tikzcd}
	\]
	admit a solution. The claim follows immediately from this fact coupled with \cite[Cor. 2.4.6.5]{HTT}.
\end{proof}

	\begin{proposition}\label{prop:holim}
	Let $\func{P_i: \mathcal{O}_i \to  \on{Set}_{\Delta}}$ with $i=1,2$, be two diagrams of simplicial sets such that
	\begin{itemize}
		\item[1)] $P_1$ is a cotower diagram such that for every $\ell \to k$ in $D_1$ the induced morphism $P_1(\ell)\to P_2(k)$ is a cofibration.
		\item[2)] $P_2$ is a pushout diagram such there exists a morphism $a \to b$ such that $P_2(a)\to P_2(b)$ is a cofibration.
	\end{itemize}
	Denote by $X_i$ the colimit of $P_i$ and by $\set{\beta_j}_{j \in \mathcal{O}_i}$ the canonical cone of $X_i$. Given $F,G \in \DD^{X_i}$ then it follows that we have equivalences of $\infty$-categories
	\[
	\Tw (\DD^X)_{(F,G)} \simeq  \holim_{j\in\mathcal{O}_i^{\op}} \Tw\left(\DD^{P_i(j)}\right)_{(\beta_j^{*}F,\beta_j^{*}G)}, \enspace \enspace \scr{L}_{(F,G)}\simeq \holim_{j\in\mathcal{O}_i^{\op}}\scr{L}_{(\beta_j^{*}F,\beta_j^{*}G)}
	\]
\end{proposition}
\begin{proof}
	Observe that the functors $\Tw(\DD^{\mathblank})$ and $\mathcal{L}_{(\mathblank)}$ preserve the ordinary limits of shape $\mathcal{O}_i$. Since taking fibers commutes with limits we observe that it is enough to show that both diagrams are injectively fibrant. This follows immediately from our hypothesis and \autoref{prop:cof}.
\end{proof}

\begin{lemma}\label{lem:innercofinal}
	Let $\iota:\Lambda^{n}_i \to \Delta^n$ be an inner horn inclusion. Then for every $F,G \in \D^{\Delta^n}$ we have equivalences of $\infty$-categories
	\[
	\func{\Tw(\DD^{\Delta^n})_{(F,G)} \to[\simeq] \Tw(\DD^{\Lambda^n_i})_{(\iota^*F,\iota^*G)},} \enspace \enspace \func{\scr{L}_{(F,G)} \to[\simeq] \scr{L}_{(\iota^*F,\iota^*G)}.}
	\]
\end{lemma}
\begin{proof}
	First, let us observe that $\DD^{\Delta^n} \to  \DD^{\Lambda_i^n}$ is a trivial fibration in the scaled model structure. After noticing this, the result follows immediately for $\Tw$. To show the claim for the second functor we just need to show that the inclusion $\iota:\Tw(\Lambda^n_i) \to \Tw(\Delta^n)$ is cofinal. Then the result will follow from the fact that restriction along $\iota^{\op}$ preserves limits. We left as an exercise to the reader this last check, that follows easily from Quillen's Theorem A. 
\end{proof}

\begin{proposition}\label{prop:reduction}
	Suppose the map $\Theta_X$ in \autoref{thm:nat} is an equivalence of Cartesian fibrations for $X=\Delta^n$ with $n=0,1$. Then for every $X \in \on{Set}_{\Delta}$ the map $\Theta_X$ is an equivalence of Cartesian fibrations.
\end{proposition}
\begin{proof}
	We will say that a simplicial set $X$ satisfies the property $(\divideontimes)$ if $\Theta_X$ is an equivalence of Cartesian fibrations. First we will assume that the simplicial sets $\Delta^n$ with $n\geq 0$ satisfy $(\divideontimes)$. 
	
	As a direct consequence of \autoref{prop:holim} 2), we deduce that boundaries $\partial \Delta^n$ fullfil condition $(\divideontimes)$ for $n\geq 0$. Let $X$ be an arbitrary simplicial set. We claim that given $n\geq 0$ the $n$-skeleton $\on{sk}_n(X)$ satisfies $(\divideontimes)$. It is clear that the claim holds for $\on{sk}_0(X)$ since it is just a disjoint union of points. Suppose that the claim holds for $\on{sk}_{l-1}(X)$  and let $I$ be the set of non degenerate simplices contained in $\on{sk}_{l}(X)\setminus \on{sk}_{l-1}(X)$. Given $i \in I$ we can attach that non-degenerate simplex via a pushout square
	\[
	\begin{tikzcd} 		
	\partial \Delta^{l} \arrow[r] \arrow[d] & \Delta^{l} \arrow[d] \\ 		
	\on{sk}_{l-1}(X) \arrow[r] & P \arrow[ul, phantom, "\ulcorner", very near start] \end{tikzcd}
	\]
	\autoref{prop:holim} implies that $\Theta_{P}$ is an equivalence.  Now let us pick a linear order on $I$ and attach one by one all the simplices in $I$. We can then produce a functor
	\[
	\func{P: I \to \on{Set}_{\Delta}}, \enspace \text{such that }\colim_{I}P \isom \on{sk}_{l}(X).
	\]
	which is an instance of \autoref{prop:holim} 1) and therefore the inductive step is proved. The same proposition now applied to $X \isom \colim_\N \on{sk}_n(X)$ finally shows that $\Theta_X$ is an equivalence of Cartesian fibrations provided $\Theta_{\Delta^n}$ is an equivalence for $n\geq 0$.
	
	We will use again induction to show that $\Theta_{\Delta^n}$ is an equivalence for $n\geq 0$. Our ground cases are $n=0,1$. Now assume the claim holds for $(n-1)\geq 1$ and pick an inner horn inclusion $\iota: \Tw(\Lambda^n_i) \to \Delta$. Then we have a commutative diagram
	\[
	\begin{tikzcd}[ampersand replacement=\&]
	\Tw(\DD^{\Delta^n})_{(F,G)} \arrow[d,"\simeq"] \arrow[r] \& \left(\scr{L}_{\Delta^n}\right)_{(F,G)} \arrow[d,"\simeq"] \\
	\Tw(\DD^{\Lambda^n_i})_{(\iota^{*}F,\iota^{*}G)} \arrow[r,"\simeq"] \& \left(\scr{L}_{\Lambda^n_i}\right)_{(\iota^{*}F,\iota^{*}G)} 
	\end{tikzcd}
	\]
	where the vertical morphisms are equivalences due to \autoref{lem:innercofinal}. It is easy to see that the bottom horizontal morphism is an equivalence due to the induction hypothesis. The result follows from 2-out-of-3.
\end{proof}

At this point we have made a drastic reduction in complexity and we are left to show that the object $\Delta^1$ satisfies $(\divideontimes)$, the case of $\Delta^0$ being obvious. We will tackle this last case by a direct computational approach. Before diving into the proof of \autoref{thm:nat} we will take a small detour to analyze the relevant combinatorics. Throughout the rest of this section we will use the coordinates $a \leq b$ for $\Delta^1$ instead of the standard $0\leq 1$ notation. In a similar fashion, we the coordinates of $\Tw(\Delta^1)$  by $ab \to aa$, $ab \to bb$.

\begin{definition}
	We define a cosimplicial object 
	\[
	\func{\scr{R}: \Delta \to \on{Set}^{\on{sc}}_{\Delta}; [n] \mapsto (\scr{R}(n),T)}, 
	\]
	\[
	\scr{R}(n)=(\Delta^n \times \Delta^1)\star (\Delta^n \times \Delta^1)^{\op}\coprod_{\Delta^{2n+1}} (\Delta^n \times \Delta^1)\star (\Delta^n \times \Delta^1)^{\op}
	\]
	We describe the scaling using the notation of \autoref{rmk:RnNotation}. $T$ is the scaling which is (1) identical on the two summands and (2) such that the non-degenerate thin 2-simplices of the first summand $(\Delta^n \times \Delta^1)\star (\Delta^n \times \Delta^1)^{\op}$ are those $\sigma$ such that
	\begin{itemize}
		\item $\sigma$ factors through either $(\Delta^n\times \Delta^1)$ or $(\Delta^n\times \Delta^1)^\op$.
		\item $i_p<j_q<k_r$ is a simplex in $\Delta^n\times \Delta^1$, and $\sigma=(i_p<j_q<\overline{k_r})$. 
		\item $k_r<j_q<i_p$ is a simplex in $\Delta^n\times \Delta^1$ and $\sigma=(i_p<\overline{j_q}<\overline{k_r})$.
		\item $i\leq j\leq k$ is a simplex of $\Delta^n$ and 
		\begin{itemize}
			\item $\sigma=i_{ab} <j_{aa}<\overline{k_a}$;
			\item $\sigma=k_{ab}<\overline{j_{aa}}<\overline{i_{ab}}$; 
			\item $\sigma=i_{aa}<j_{aa}<\overline{k_{ab}}$;
			\item $\sigma=k_{ab}<\overline{j_{aa}}<\overline{i_{aa}}$; 
			\item $\sigma=i_{ab}<j_{ab}<\overline{k_{aa}}$; or
			\item $\sigma=k_{aa}<\overline{j_{ab}}<\overline{i_{ab}}$. 
		\end{itemize} 
	\end{itemize}
\end{definition}

\begin{remark}\label{rmk:RnNotation}
	We can describe the underlying simplicial set of $\scr{R}(n)$ as the nerve of a poset $R_n$ as follows
	\begin{itemize}
		\item The set of objects is given by symbols $\ell_{\epsilon}$ where $\ell \in [n]$ and $\epsilon \in \set{ab,aa,bb}$ together with their formal duals $\overline{\ell}_{\epsilon}$.
		\item We declare $\ell_{ab} \leq k_{\epsilon}$ where $\epsilon \in \set{ab,aa,bb}$ if and only if $\ell \leq k$. Dually we declare $\overline{\ell}_{ab} \leq \overline{k}_{\epsilon}$ if and only if $k \leq \ell$. Finally we declare $\ell_{\epsilon} < \overline{\ell}_{\epsilon}$. The ordering on $R_n$ is the minimal one generated by the inequalities above.
	\end{itemize} 
	We provide graphical representations of the posets for $n\leq 2$:
	\begin{figure}[h]
		\centering
		\begin{tabular}{l l|l}
			$R_0$ & & 
			\begin{tikzpicture}[baseline={([yshift=-.7ex]current bounding box.center)},mybox/.style={ inner sep=15pt}]
			\node[mybox] (box){%
				\begin{tikzcd}
				\overline{0}_{aa} \arrow[rr] &  & \overline{0}_{ab}                       &  & \overline{0}_{bb} \arrow[ll] \\
				&  &                                         &  &                              \\
				0_{aa} \arrow[uu]            &  & 0_{ab} \arrow[rr] \arrow[ll] \arrow[uu] &  & 0_{bb} \arrow[uu]           
				\end{tikzcd}
			};
			\end{tikzpicture}\\\hline
			$R_1$ & & 
			\begin{tikzpicture}[baseline={([yshift=-.7ex]current bounding box.center)},mybox/.style={ inner sep=15pt}]
			\node[mybox,scale=0.8] (box){%
				\begin{tikzcd}
				&  & \overline{0}_{aa} \arrow[rrrr]            &  &                                                         &  & \overline{0}_{ab}            &  &                               &  & \overline{0}_{bb} \arrow[llll]            \\
				0_{aa} \arrow[rru] \arrow[dd] &  &                                           &  & 0_{ab} \arrow[dd] \arrow[rrrr] \arrow[rru] \arrow[llll] &  &                              &  & 0_{bb} \arrow[rru] \arrow[dd] &  &                                           \\
				&  & \overline{1}_{aa} \arrow[rrrr] \arrow[uu] &  &                                                         &  & \overline{1}_{ab} \arrow[uu] &  &                               &  & \overline{1}_{bb} \arrow[uu] \arrow[llll] \\
				1_{aa} \arrow[rru]            &  &                                           &  & 1_{ab} \arrow[rru] \arrow[rrrr] \arrow[llll]            &  &                              &  & 1_{bb} \arrow[rru]            &  &                                          
				\end{tikzcd}
			};
			\end{tikzpicture}\\\hline
			$R_2$ & & 
			\begin{tikzpicture}[baseline={([yshift=-.7ex]current bounding box.center)},mybox/.style={ inner sep=15pt}]
			\node[mybox,scale=0.8] (box){%
				\begin{tikzcd}
				&  & \overline{0}_{aa} \arrow[rrrr]            &  &                                                         &  & \overline{0}_{ab}            &  &                               &  & \overline{0}_{bb} \arrow[llll]            \\
				0_{aa} \arrow[rru] \arrow[dd] &  &                                           &  & 0_{ab} \arrow[dd] \arrow[rrrr] \arrow[rru] \arrow[llll] &  &                              &  & 0_{bb} \arrow[rru] \arrow[dd] &  &                                           \\
				&  & \overline{1}_{aa} \arrow[rrrr] \arrow[uu] &  &                                                         &  & \overline{1}_{ab} \arrow[uu] &  &                               &  & \overline{1}_{bb} \arrow[uu] \arrow[llll] \\
				1_{aa} \arrow[rru] \arrow[dd] &  &                                           &  & 1_{ab} \arrow[rru] \arrow[rrrr] \arrow[llll] \arrow[dd] &  &                              &  & 1_{bb} \arrow[rru] \arrow[dd] &  &                                           \\
				&  & \overline{2}_{aa} \arrow[rrrr] \arrow[uu] &  &                                                         &  & \overline{2}_{ab} \arrow[uu] &  &                               &  & \overline{2}_{bb} \arrow[llll] \arrow[uu] \\
				2_{aa} \arrow[rru]            &  &                                           &  & 2_{ab} \arrow[llll] \arrow[rru] \arrow[rrrr]            &  &                              &  & 2_{bb} \arrow[rru]            &  &                                          
				\end{tikzcd}
			};
			\end{tikzpicture}\\\hline
		\end{tabular}
		\caption{The posets $R_n$ for $n \le 2$.}
		\label{fig:posets}
	\end{figure}
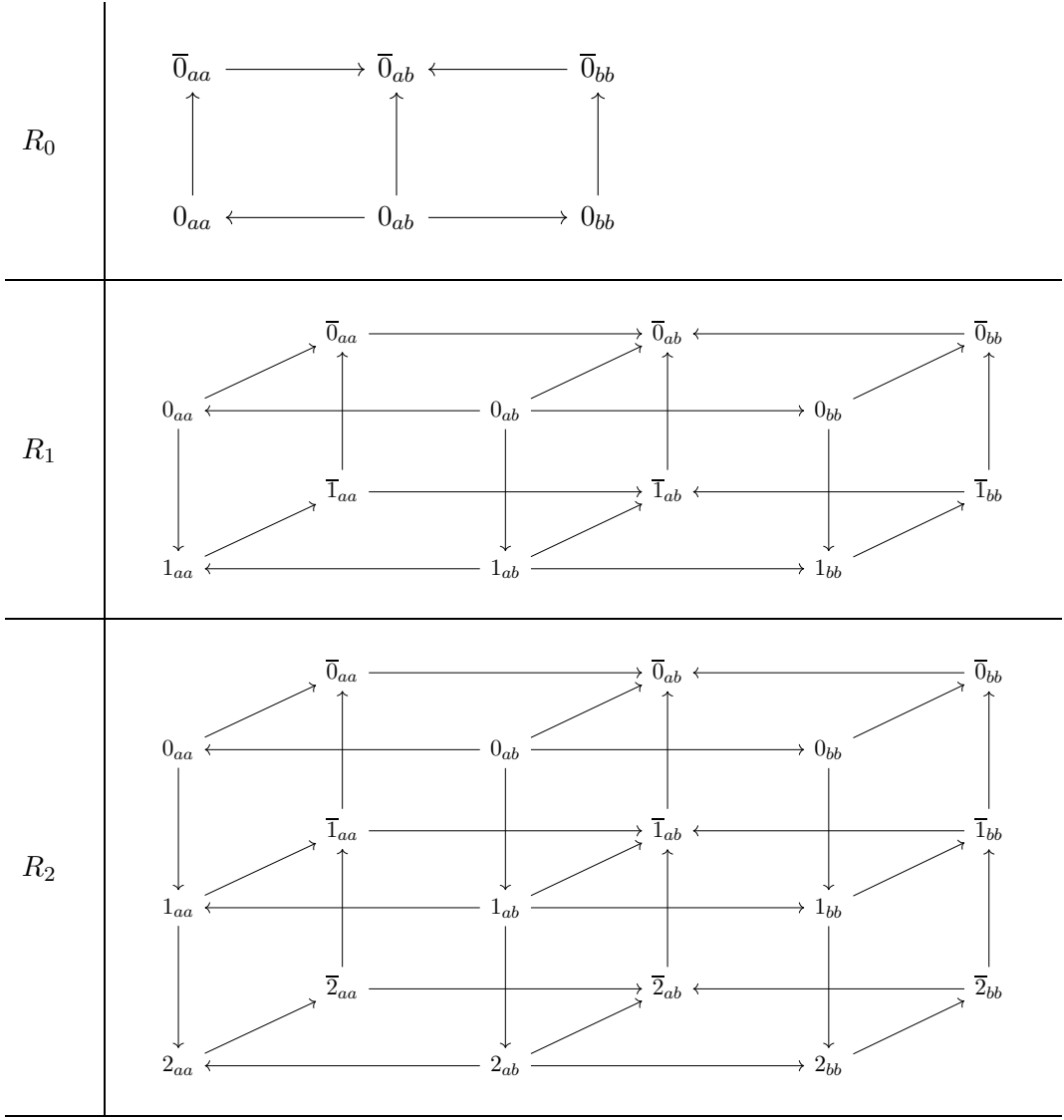
\end{remark}

\begin{remark}
	We observe that the posets above come equipped with an isomorphism $(R_n)^{\op}\isom R_n$ given by applying the “bar operator” $\overline{(\mathblank)}$. It is worth pointing out that our scaling is symmetric with respect to this duality.
\end{remark}

\begin{definition}
	We define a cosimplicial object
	\[
	\func{\scr{Q}:\Delta \to \on{Set}^{\on{sc}}_{\Delta}; [n] \mapsto Q(\Delta^n \times \Tw(\Delta^1)^{\sharp}) }
	\]
	where $Q$ was already introduced in \autoref{defn:CosimpQ} and the functoriality is the obvious one. Since $Q$ preserves colimits we see that $\scr{Q}(n)$ splits into 
	\[
	Q(\Delta^n \times \Tw(\Delta^1))  \isom Q(\Delta^n \times \Delta^1) \coprod\limits_{Q(\Delta^n)}Q(\Delta^n \times \Delta^1).
	\]
\end{definition}

\begin{remark} \label{rmk:Qsimps}
	Recall that our definitions imply that a map $\scr{Q}(n) \to \DD$ corresponds precisely to a functor $\func{ \Delta^n \times \Tw(\Delta^1) \to \Tw (\DD).}$ We see that a simplex in $\scr{L}_{(F,G)}$ is given by a map $\scr{Q}(n)\to \DD$ satisfying the obvious conditions after restriction to  $\Delta^n \times \Tw(\Delta^1), (\Delta^n \times \Tw(\Delta^1))^{\op} \subset Q(n)$.
\end{remark}

\begin{definition}
	Let $n\geq 0$ and observe that $\scr{R}(n)$ fits into a cocone for the colimit defining $\scr{Q}(n)$. Then the induced cofibrations $\func{\epsilon_n: \scr{Q}(n) \to \scr{R}(n)}$, assemble into map of cosimplicial objects $\func{\xi: \scr{Q}\nat \scr{R} }$. 
\end{definition}

\begin{definition}
	We define a cosimplicial object
	\[
	\func{\scr{T}: \Delta \to \on{Set}_{\Delta}^{\on{sc}}; [n] \mapsto Q(n)\times \Delta^1.}
	\]
\end{definition}

\begin{remark}
	Analogously to \autoref{rmk:Qsimps}, we can identify a simplex $\Delta^n \to \Tw (\DD^{\Delta^1})_{(F,G)}$ with a map $\scr{T}(n) \to \DD$ such that the restrictions to $\Delta^n \times \Delta^1$ and $(\Delta^n)^{\op}\times \Delta^1$ are constant on $F$ and $G^{\op}$ respectively.
\end{remark}

\begin{definition}
	Define a map of posets
	\[
	\func{\mu_n: \scr{T}(n) \to \scr{R}(n);(\ell,a) \mapsto \ell_{ab}; (\overline{\ell},a) \mapsto \overline{\ell}_{aa} ;(\ell,b) \mapsto \ell_{bb} ; (\overline{\ell},b)\mapsto \overline{\ell}_{ab} }
	\] 
	Then the maps $\mu_n$ assemble into a map of cosimplicial objects $\func{\mu: \scr{T} \nat \scr{R}}$.
\end{definition}

\begin{remark}
	At this juncture it is worth noting that the scaling on $\scr{R}(n)$ is the minimal scaling such that $\func{\xi:\scr{Q}\nat \scr{R}}$ and $\func{\mu:\scr{T}\nat \scr{R}}$ respect the scaling, and such that the scaling on $\scr{R}(n)$ has the two symmetries previously mentioned. 
\end{remark}

Let us take a small break to put the previous definitions into perspective. We have defined three cosimplicial objects $\scr{R},\scr{Q}$ and $\scr{T}$, the last two of which define the simplices of the $\infty$-categories $\Tw(\DD^{\Delta^1})_{(F,G)}$ and $\scr{L}_{(F,G)}$ respectively. The proof of \autoref{thm:nat} will rely on identifying $\scr{R}$ as an interpolating cosimplicial object between $\scr{Q}$ and $\scr{T}$. In the next proposition, we will show an equivalence of cosimplicial objects between $\scr{Q}$ and $\scr{R}$ thus providing a key technical ingredient for the proof of the main theorem. Readers unwilling to join us for this combinatorial ride can safely skip the next proof.

\begin{proposition}\label{prop:scaledanodyne}
	The map of cosimplicial objects $\func{\xi: \scr{Q} \nat \scr{R}}$ is a levelwise trivial cofibration in the scaled model structure.
\end{proposition}
\begin{proof}
	We will prove something stronger, namely, for every $n \geq 0$ the map $\xi_n$ is scaled anodyne. Using the description of both $\scr{Q}(n)$ and $\scr{R}(n)$ as pushouts, we deduce that it will suffice to show that the map
	\[
	\func{Q(\Delta^n \times (\Delta^1)^{\sharp}) \to \left((\Delta^n \times \Delta^1)\star (\Delta^n \times \Delta^1)^{\op}\right)_{\diamond}}
	\]
	is scaled anodyne, where the subscript $\diamond$ indicates the scaling induced by that of $\scr{R}(n)$. Before embarking upon the proof of our claim we will set some notation
	\[
	Q(\Delta^n \times (\Delta^1)^{\sharp})= A^n_{\diamond}, \enspace \enspace  \left((\Delta^n \times \Delta^1)\star (\Delta^n \times \Delta^1)^{\op}\right)_{\diamond}=B^n_{\diamond}.
	\]
	Let $(r,s)$ be a pair of non-negative integers such that $r,s \leq n$. We define a simplex
	\[
	\func*{\sigma_{(r,s)}:\Delta^{2n+3} \to B^n_{\diamond}; \ell \mapsto \begin{cases}
		\enspace \ell_a \enspace \text{ if } \ell \leq r \\
		\enspace \ell_{b} \enspace \text{ if } r+1 \leq \ell \leq n+1 \\
		\enspace \overline{\ell}_b \enspace \text{ if } n+2 \leq \ell \leq 2n+2-s \\
		\enspace \overline{\ell}_a \enspace  \text{ if } 2n+3-s \leq \ell \leq 2n+3 \\	
		\end{cases}}
	\]
	and note that $B^{n}_{\diamond}=\bigcup\limits_{(r,s)} \sigma_{(r,s)}$. We further divide the simplices $\sigma_{(r,s)}$ into three families parametrized by $r-s=\alpha$. To illuminate our claims let us include some examples for $n=3$.
	\\
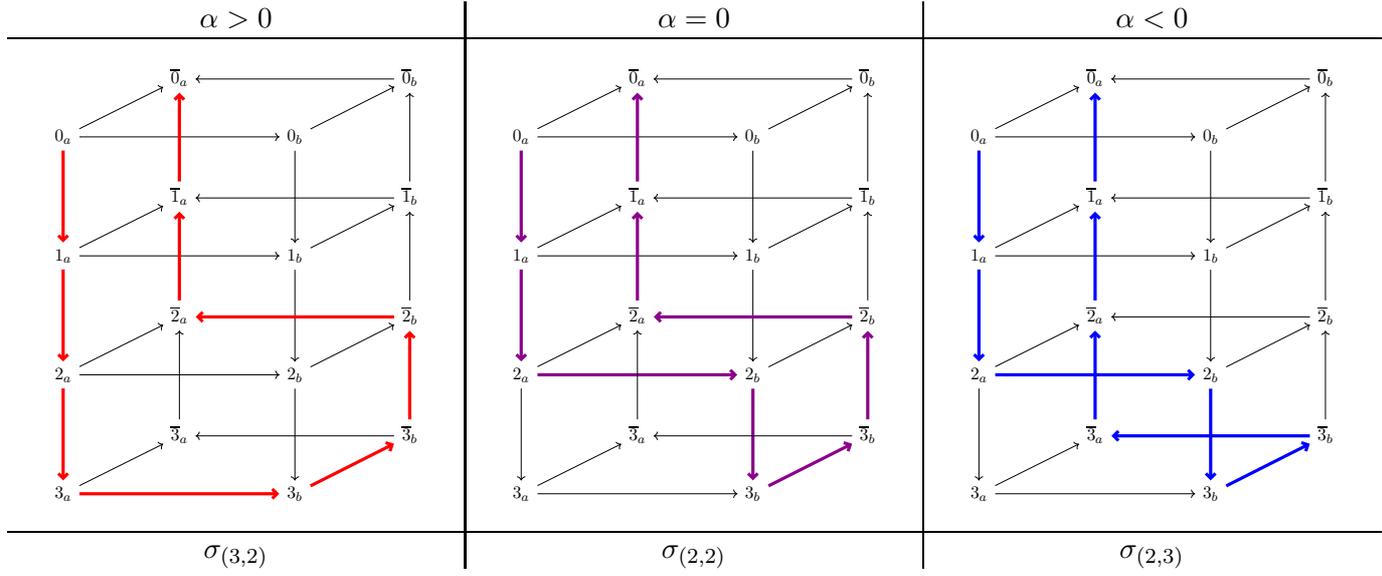
	\begin{figure}[h]
		\centering
		\begin{tabular}{c | c | c}
			$\alpha>0$ &  $\alpha=0$ &  $\alpha<0$ \\ \hline
			
			\begin{tikzpicture}[baseline={([yshift=-.7ex]current bounding box.center)},mybox/.style={ inner sep=15pt}]
			\node[mybox,scale=0.6] (box){%
				\begin{tikzcd}
				&  & \overline{0}_a            &  &                            &  & \overline{0}_b \arrow[llll]            \\
				0_a \arrow[dd,boldred] \arrow[rrrr] \arrow[rru] &  &                           &  & 0_b \arrow[rru] \arrow[dd] &  &                                        \\
				&  & \overline{1}_a \arrow[uu,boldred] &  &                            &  & \overline{1}_b \arrow[uu] \arrow[llll] \\
				1_a \arrow[rru] \arrow[rrrr] \arrow[dd,boldred] &  &                           &  & 1_b \arrow[rru] \arrow[dd] &  &                                        \\
				&  & \overline{2}_a \arrow[uu,boldred] &  &                            &  & \overline{2}_b \arrow[uu] \arrow[llll,boldred] \\
				2_a \arrow[rru] \arrow[rrrr] \arrow[dd,boldred] &  &                           &  & 2_b \arrow[rru] \arrow[dd] &  &                                        \\
				&  & \overline{3}_a \arrow[uu] &  &                            &  & \overline{3}_b \arrow[llll] \arrow[uu,boldred] \\
				3_a \arrow[rru] \arrow[rrrr,boldred]            &  &                           &  & 3_b \arrow[rru,boldred]            &  &                                       
				\end{tikzcd}
			};
			\end{tikzpicture}  &  \begin{tikzpicture}[baseline={([yshift=-.7ex]current bounding box.center)},mybox/.style={ inner sep=15pt}]
			\node[mybox,scale=0.6] (box){%
				\begin{tikzcd}
				&  & \overline{0}_a            &  &                            &  & \overline{0}_b \arrow[llll]            \\
				0_a \arrow[dd,boldmagenta] \arrow[rrrr] \arrow[rru] &  &                           &  & 0_b \arrow[rru] \arrow[dd] &  &                                        \\
				&  & \overline{1}_a \arrow[uu,boldmagenta] &  &                            &  & \overline{1}_b \arrow[uu] \arrow[llll] \\
				1_a \arrow[rru] \arrow[rrrr] \arrow[dd,boldmagenta] &  &                           &  & 1_b \arrow[rru] \arrow[dd] &  &                                        \\
				&  & \overline{2}_a \arrow[uu,boldmagenta] &  &                            &  & \overline{2}_b \arrow[uu] \arrow[llll,boldmagenta] \\
				2_a \arrow[rru] \arrow[rrrr,boldmagenta] \arrow[dd] &  &                           &  & 2_b \arrow[rru] \arrow[dd,boldmagenta] &  &                                        \\
				&  & \overline{3}_a \arrow[uu] &  &                            &  & \overline{3}_b \arrow[llll] \arrow[uu,boldmagenta] \\
				3_a \arrow[rru] \arrow[rrrr]            &  &                           &  & 3_b \arrow[rru,boldmagenta]            &  &                                       
				\end{tikzcd}
			};
			\end{tikzpicture} & \begin{tikzpicture}[baseline={([yshift=-.7ex]current bounding box.center)},mybox/.style={ inner sep=15pt}]
			\node[mybox,scale=0.6] (box){%
				\begin{tikzcd}
				&  & \overline{0}_a            &  &                            &  & \overline{0}_b \arrow[llll]            \\
				0_a \arrow[dd,boldblue] \arrow[rrrr] \arrow[rru] &  &                           &  & 0_b \arrow[rru] \arrow[dd] &  &                                        \\
				&  & \overline{1}_a \arrow[uu,boldblue] &  &                            &  & \overline{1}_b \arrow[uu] \arrow[llll] \\
				1_a \arrow[rru] \arrow[rrrr] \arrow[dd,boldblue] &  &                           &  & 1_b \arrow[rru] \arrow[dd] &  &                                        \\
				&  & \overline{2}_a \arrow[uu,boldblue] &  &                            &  & \overline{2}_b \arrow[uu] \arrow[llll] \\
				2_a \arrow[rru] \arrow[rrrr,boldblue] \arrow[dd] &  &                           &  & 2_b \arrow[rru] \arrow[dd,boldblue] &  &                                        \\
				&  & \overline{3}_a \arrow[uu,boldblue] &  &                            &  & \overline{3}_b \arrow[llll,boldblue] \arrow[uu] \\
				3_a \arrow[rru] \arrow[rrrr]            &  &                           &  & 3_b \arrow[rru,boldblue]            &  &                                       
				\end{tikzcd}
			};
			\end{tikzpicture} \\ \hline
			$\sigma_{(3,2)}$ & $\sigma_{(2,2)}$ & $\sigma_{(2,3)}$
		\end{tabular}
	\end{figure}

	We define $B^{+}_{\diamond}$ (resp. $B^{-}_{\diamond}$, resp. $B^{0}_{\diamond}$) as the union of the simplices $\sigma_{(r,s)}$ such that $\alpha \geq 0$ (resp. $\alpha \leq 0$, resp. $\alpha=0$) with the induced scaling. It follows from unwinding the definitions that $A^{n}_{\diamond}=B^{0}_{\diamond}$ and that $B^{+}_{\diamond} \cap B^{-}_{\diamond} = B^{0}_{\diamond}$. We have thus produced a pushout square
	\[
	\begin{tikzcd}[ampersand replacement=\&]
	A^n_{\diamond} \arrow[r] \arrow[d] \& B^{+}_{\diamond} \arrow[d] \\
	B^{-}_{\diamond} \arrow[r] \& B^{n}_{\diamond}. \arrow[ul,phantom,"\ulcorner",very near start]
	\end{tikzcd}
	\]
	We turn now to show that $A^n_{\diamond} \to B^{\pm}_{\diamond}$ is scaled anodyne. First let us tackle the case $\alpha>0$. To this end we produce a filtration 
	\[
	\func{A^n_{\diamond}=X_0 \to X_1 \to \cdots\to  X_{n-1}\to X_n=B^{+}_{\diamond}}
	\]
	where $X_j$ is the scaled simplicial subset consisting in those simplices contained in some $\sigma_{(r,s)}$ with $\alpha \leq j$. We claim that in order to show that $X_{j-1} \to X_j$ is scaled anodyne it suffices to show that top horizontal  morphism $f_{(r,s)}$ in the pullback diagram below
	\[
	\begin{tikzcd}[ampersand replacement=\&]
	W_{(r,s)} \arrow[r,"f_{(r,s)}"] \arrow[dr,phantom, "\lrcorner", very near start] \arrow[d] \& \Delta^{2n+3} \arrow[d,"\sigma_{(r,s)}"] \\
	X_{j-1} \arrow[r] \& X^{+}_{j}.
	\end{tikzcd}
	\]
	is scaled anodyne with respect to the induced scaling. Indeed, we observe that given $(r,s),(u,v)$ such that $r-s=u-v=j$ then it follows that $\sigma_{(r,s)}\cap \sigma_{(u,v)}\in X_{j-1}$ and the claim follows. After some contemplation we discover that
	\[
	W_{(r,s)}=d_{r}(\sigma_{(r,s)}) \cup d_{2n+2-s}(\sigma_{(r,s)}).
	\]
	Consequently we can define a dull subset consisting of the sets $\set{r}$, $\set{2n+2-s}$ with pivot point $2n+2-r$. Using \autoref{lem:pivot} we conclude that $A^n_{\diamond} \to B^{+}_{\diamond}$ is scaled anodyne.
	
	The case $\alpha<0$ is a formal dual of the case just proved. To see this we observe that the duality on $\scr{R}(n)$ restricts to $(B^{+}_{\diamond})^{\op}\isom B^{-}_{\diamond}$ and that our scaling is symmetric. The case $\alpha<0$ follows, concluding the proof.
\end{proof}

\begin{corollary}\label{cor:interpolation1}
	Let $\DD$ be an $\infty$-bicategory and let $\scr{X}$ be the simplicial set obtained via the cosimplicial object $\scr{R}$. Consider the induced map
	\[
	\func{\xi^{*}: \scr{X} \to \scr{L}_{\Delta^1}.}
	\]
	Then the $\xi^{*}$ is a trivial Kan fibration. In particular, after passing to fibers we obtain an equivalence of $\infty$-categories
	\[
	\func{\scr{X}_{(F,G)} \to[\simeq]  \scr{L}_{(F,G)}.}
	\]
\end{corollary}
\begin{proof}
	Note that as an immediate consequence of \autoref{prop:scaledanodyne} we obtain an scaled anodyne map $\partial \scr{Q}^n \to \partial \scr{R}^{n}$. Consider the morphisms
	\[
	\func{\scr{Q}(n) \to \scr{Q}(n)\coprod\limits_{\partial \scr{Q}^n} \partial \scr{R}^{n} \to \scr{R}(n),}
	\]
	and note the last map is a trivial cofibration by 2-out-of-3. The reader will observe that the boundary lifting problems are in bijection with lifting problems of the form
	\[
	\begin{tikzcd}[ampersand replacement=\&]
	\scr{Q}(n)\coprod\limits_{\partial \scr{Q}^n} \partial \scr{R}^{n} \arrow[d] \arrow[r] \& \DD \\
	\scr{R}(n) \arrow[ur,dotted]
	\end{tikzcd}
	\]
	and hence the result.
\end{proof}

\begin{figure}[htb]
	\begin{center}
		\begin{tikzpicture}[mybox/.style={ inner sep=15pt}]
		\node[mybox,scale=0.8] (box){%
			\begin{tikzcd}
			&  & \overline{0}_{aa} \arrow[rrrr,boldblue]            &  &                                                         &  & \overline{0}_{ab}            &  &                               &  & \overline{0}_{bb} \arrow[llll]            \\
			0_{aa} \arrow[rru] \arrow[dd] &  &                                           &  & 0_{ab} \arrow[dd,boldblue]\arrow[ull,boldblue] \arrow[rrrr,boldblue] \arrow[rru,boldblue] \arrow[llll] &  &                              &  & 0_{bb} \arrow[rru] \arrow[dd,boldblue] \arrow[ull,boldblue] &  &                                           \\
			&  & \overline{1}_{aa} \arrow[rrrr,boldblue] \arrow[uu,boldblue] &  &                                                         &  & \overline{1}_{ab} \arrow[uu,boldblue] &  &                               &  & \overline{1}_{bb} \arrow[uu] \arrow[llll] \\
			1_{aa} \arrow[rru]            &  &                                           &  & 1_{ab} \arrow[rru,boldblue]\arrow[ull,boldblue] \arrow[rrrr,boldblue] \arrow[llll]            &  &                              &  & 1_{bb} \arrow[rru] \arrow[ull,boldblue]            &  &                                          
			\end{tikzcd}
		};
		\end{tikzpicture}
	\end{center}
	\caption{$\scr{T}(1)$ pictured in blue as a subset of $\scr{R}(1)$ under the inclusion $\mu_1$. The map $\psi_1$ can be alternately characterized as the unique map such that $\psi_1\circ \mu_1=\id$ and $\psi$ preserves $\overline{(-)}$ and its dual.} \label{fig:TinR}
\end{figure}
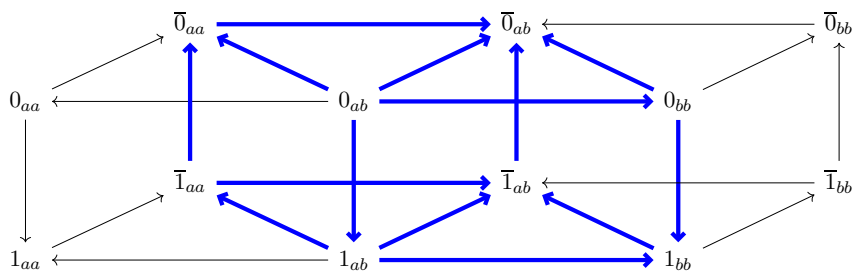

\begin{construction}
	We define a map 
	\[
	\func{\scr{R}(n)\to Q(n)}
	\]
	by requiring $\func{\overline{i_{xy}}\mapsto \overline{i}}$ and $\func{i_{xy}\mapsto i}$. We further define a map 
	\[
	\func{\scr{R}(n)\to \Delta^1}
	\]
	by 
	$
	\func{i_{xy}\mapsto x; \overline{i_{xy}}\mapsto y.}
	$
	Note that both of these maps can be easily checked to preserve the scalings. Together, they thus define  a map 
	$
	\func{\psi_n:\scr{R}(n)\to \scr{T}(n)}
	$
	such that $\psi_n\circ \mu_n=\id$ (see \autoref{fig:TinR}). Moreover, the $\psi_n$ yield a natural transformation $\psi:\scr{R}\to \scr{T}$. We denote by $\func{\psi^\ast:\Tw(\DD^{\Delta^1})\to \scr{X}}$ the induced map.  
\end{construction}


\begin{lemma}\label{lem:interpolation2}
	The diagram 
	\[
	\begin{tikzcd}[ampersand replacement=\&]
	\Tw(\DD^{\Delta^1}) \arrow[dr,"\Theta_{\Delta^1}"] \arrow[d,swap,"\psi_*"]	\&  \\
	\scr{X} \arrow[r,"\xi^\ast"'] \& \scr{L}_{\Delta^1}
	\end{tikzcd}
	\]
	commutes. 
\end{lemma}

\begin{proof}
	Left as an exercise to the reader.
\end{proof}

\begin{proof}[Proof of \autoref{thm:nat}]
	By virtue of \autoref{prop:reduction}, it will suffice to show that $\Theta_X$ is an equivalence of Cartesian fibrations for $X=\Delta^n$ with $n=0,1$. The case $n=0$ is obvious. To show the case $n=1$ we observe that due to \autoref{cor:interpolation1} and \autoref{lem:interpolation2} it will suffice to show $\psi_{*}$ is an equivalence of $\infty$-categories upon passage to fibers. We further note that since $\mu^* \circ \psi^*= \id$ it will be enough to show that $\varphi^*=\psi^* \circ \mu^*$ is a fiberwise equivalence. Let $\sigma: \Delta^n \to \Delta^1$ and let $j\in [n]$ be the first object such that $\sigma(j)=1$ if $\sigma$ is constant on $0$ we set the convention $j=n+1$. Now we can define a map of scaled simplicial sets 
	\[
	\func{\varphi^1_\sigma: \scr{R}(n) \to \scr{R}(n)}
	\]
	which leaves every object invariant except those of the form $\ell_{aa}$ with $\ell < j$ which are sent to $\ell_{ab}$. Given $\rho:\Delta^n \to \scr{X}_{(F,G)}$ we define a simplex $H(\sigma,\rho):\Delta^n \to \scr{X}_{(F,G)}$ given by the composite
	\[
	\func{ \scr{R}(n) \to[\varphi_{\sigma}] \scr{R}(n) \to[\rho] \DD}
	\]
	This assignment extends to a homotopy $\func{H_1: \Delta^1 \times \scr{X}_{(F,G)} \to \scr{X}_{(F,G)}}$ which is component-wise an equivalence. This exhibits an equivalence of morphisms $\id\sim \left(\varphi^1_0\right)_{(F,G)}^{*}$ where $\varphi^1_0$ denotes the previously defined map with respect with the constant simplex at $0$.
	
	Let $\sigma: \Delta^n \to \Delta^1$. Then we define a map of scaled simplicial sets
	\[
	\func{\varphi^2_{\sigma}:\scr{R}(n) \to \scr{R}(n)}
	\]
	that leaves every object invariant except those of the form $\ell_{aa}$ which are sent to $\ell_{ab}$ and those of the form $\overline{\ell}_{bb}$ with $\ell < j$ with are sent to $\overline{\ell}_{ab}$. We can now define, in perfect analogy to the situation above, a natural equivalence $\func{H_2: \Delta^1 \times \scr{X}_{(F,G)} \to \scr{X}_{(F,G)}}$ between $\varphi^{*}_{(F,G)}$ and $\left(\varphi^1_0\right)_{(F,G)}^{*}$, hence the result.
\end{proof}

\subsection{Application: weighted colimits of $\infty$-categories}

We conclude this section (and thereby the paper) with several corollaries of \autoref{thm:nat}, and their application to the 2-dimensional universal property of weighted colimits.  Because of the technical complexities shunted into the proofs of the properties of $\Tw(\DD)$, the proof of this 2-universal property is extremely straightforward.

Throughout this section we will fix an $\infty$-category $\C$ and a pair of functors $\func{F: \C \to \iCat_{\infty}}$, $\func{W: \C^{\op} \to \iCat_{\infty}}$ that we will refer of as the diagram and the weight functors respectively. We will denote by $\goth{C}\!\on{at}_{\infty}$ the $\infty$-bicategory of $\infty$-categories.

\begin{definition}
	Let $\DD$ be an $\infty$-bicategory. We say that the underlying $\infty$-category $\D$ is \emph{tensored} over $\iCat_{\infty}$ with respect to $\DD$ if for every $d \in \D$ the mapping functor $\on{Map}_{\DD}(d,\mathblank)$ has a left adjoint $\mathblank \tensor d: \iCat_{\infty} \to \D$; in this case these adjoints determine an essentially unique functor $\iCat_{\infty}\times \D \to \D$.
\end{definition}

\begin{corollary}\label{cor:tensored}
	Let $\DD$ be an $\infty$-bicategory such that the underlying $\infty$-category $\D$ is tensored over $\iCat_{\infty}$ with respect to $\DD$. Then for every $\infty$-category $\C$ the functor category $\D^{\C}$ is tensored over $\iCat_\infty$ with respect to $\DD^{\C}$.
\end{corollary}
\begin{proof}
	Combine \autoref{thm:nat} with  \cite[Lem. 6.7]{GHN}.
\end{proof}

\begin{corollary}\label{cor:straight}
	Let $\C$ be an $\infty$-category and let $\func{\scr{E} \to \C}$, $\func{\scr{E}' \to \C}$ be Cartesian fibrations. We denote by $\on{Fun}^{\on{cart}}_{\C}(\scr{E},\scr{E}')$ the $\infty$-category of maps of Cartesian fibrations. Then there is a natural equivalence of $\infty$-categories
	\[
	\func{\on{Fun}^{\on{cart}}_{\C}(\scr{E},\scr{E}') \to[\simeq] \on{Nat}_{\C}(\on{St}(\scr{E}),\on{St}(\scr{E}))}
	\]
	where $\on{St}$ denotes the straightening functor.
\end{corollary}
\begin{proof}
	Combine \autoref{thm:nat} with \cite[Prop. 6.9]{GHN}.
\end{proof}

\begin{remark}
	It is worth noting that \autoref{cor:straight} can be interpreted as very compelling evidence suggesting that an enhanced version of the straightening functor $\on{St}$, will yield an equivalence of $\infty$-\emph{bicategories} between the category of Cartesian fibrations over $\C$ and the category of $\iCat_{\infty}$-valued functors on $\C$.
\end{remark}

Recall that in \cite[Def. 2.7]{GHN}, the authors define the weighted colimit of $F$ with weight $W$ as the coend 
\[
\colim\limits_{\Tw(\C)}W \times F.
\]
According to this definition the universal property of the weighted colimit is purely 1-dimensional. Our aim in this section is to show that the previous definition is just a shadow of a bicategorical universal property and thus find a bridge between the classical theory of weighted colimits in 2-categories and the realm of $\infty$-bicategories.

\begin{definition}
	The weighted colimit of $F$ with weight $W$ is the unique (up to equivalence) $\infty$-category representing the functor
	\[
	\begin{tikzcd}[ampersand replacement=\&]
	\goth{C}\!\on{at}_{\infty} \arrow[r,"\mathcal{Y}"] \& \goth{C}\!\on{at}_{\infty}^{\goth{C}\!\on{at}^{\op}_{\infty}} \arrow[r,"F^{*}"] \& \goth{C}\!\on{at}_{\infty}^{\C^{\op}} \arrow[rrr,"\on{Nat}_{\C^{\op}}(W{,}-)"] \& \& \& \goth{C}\!\on{at}_{\infty}
	\end{tikzcd}
	\]
	where $\mathcal{Y}$ denotes the bicategorical Yoneda embedding.\footnote{We are here ignoring some substantial set-theoretic complexities. We should, more properly, fix a nested pair of Grothendieck universes, and consider variants of $\goth{C}\!\on{at}_{\infty}$ based on size. In the interest of concision, we will sweep such set-theoretic concerns under the rug, leaving their contemplation to the interested reader.} We will denote weighted colimit by $W\otimes F$. More compactly, this definition means that there is an equivalence $\on{Nat}_{\C^\op}(W,\Fun(F(-),\scr{X}))\simeq \Fun(W\otimes F,\scr{X})$, natural in $\scr{X}$.
\end{definition}

\begin{remark}
	This definition of weighted colimits was already considered in more generality in \cite{AG20}.
\end{remark}

\begin{theorem}
	Consider a pair of functors $\func{F: \C \to \iCat_{\infty}}$, $\func{W: \C^{\op} \to \iCat_{\infty}}$. Then there is an equivalence of $\infty$-categories
	\[
	\func{W \tensor F \to[\simeq] \colim\limits_{\Tw(\C)}W \times F.}
	\]
\end{theorem}

\begin{proof}
	Let $\scr{X}$ be an $\infty$-category. We trace out a chain of equivalences, natural in $\scr{X}$. By \autoref{thm:nat}, we have 
	\[
	\on{Nat}_{\C^\op} (W,\Fun(F(-),\scr{X}))\simeq \lim_{\Tw(\C)^\op} \Fun(W(-),\Fun(F(-),\scr{X})).
	\]
	A standard chain of manipulations then yields
	\[
	\lim_{\Tw(\C)^\op}\Fun(W(-),\Fun(F(-),\scr{X}))\simeq \lim_{\Tw(\C)^\op}\Fun(W(-)\times F(-),\scr{X}))\simeq \Fun \left(\colim_{\Tw(\C)}W\times F, \scr{X}\right)
	\]
	so that $\colim\limits_{\Tw(\C)} W\times F$ satisfies the universal property defining $W\otimes F$, completing the proof. 
\end{proof}

\end{document}